\scrollmode

\documentclass[10pt]{amsart}
\usepackage{amssymb}
\usepackage[all]{xy}
\usepackage{graphicx}
\usepackage{xcolor}
\usepackage{mathtools}
\usepackage{lmodern}
\usepackage{tcolorbox}

\theoremstyle{plain}
\newtheorem{theorem}{Theorem}
\newtheorem{coro}[theorem]{Corollary}
\newtheorem{lemma}[theorem]{Lemma}

\newtheorem{proposition}[theorem]{Proposition}
\newtheorem{ques}[theorem]{Question}

\newtheorem{thmABC}{Theorem}

\newtheorem{proABC}[thmABC]{Proposition}

\newtheorem{conjABC}[thmABC]{Conjecture}

\theoremstyle{remark}
\newtheorem{example}[theorem]{Example}
\newtheorem{remark}[theorem]{Remark}
\newtheorem*{notation}{Notation}

\numberwithin{theorem}{section}
\numberwithin{equation}{section}

\theoremstyle{definition}
\newtheorem{definition}[theorem]{Definition}

\newcommand{\F}{\mathbb{F}}
\newcommand{\Z}{\mathbb{Z}}
\newcommand{\N}{\mathbb{N}}

\newcommand{\res}{{\rm res}}

\newcommand{\sV}{{\mathcal V}}
\newcommand{\sE}{{\mathcal E}}
\newcommand{\calG}{{\mathcal G}}

\newcommand{\kernel}{\mathrm{ker}}
\newcommand{\image}{\mathrm{im}}

\newcommand{\dd}{\mathrm{d}}
\newcommand{\cdd}{\mathrm{cd}}
\newcommand{\rr}{\mathrm{r}}

\newcommand{\argu}{\hbox to 7truept{\hrulefill}}

\DeclareMathOperator{\Gal}{Gal}

\usepackage{amssymb, amsthm,amsmath,dsfont}
\usepackage{enumerate}
\usepackage{stmaryrd}
\usepackage{mathrsfs}
\PassOptionsToPackage{hyphens}{url} \usepackage{hyperref}
\usepackage[all]{xy}
\usepackage{soul}
\usepackage{cancel}
\usepackage{qtree}

\newcommand{\norm}[1]{\vert #1 \vert}
\renewcommand{\a}{\alpha}
\renewcommand{\b}{\beta}
\renewcommand{\c}{\gamma}
\renewcommand{\d}{\delta}

\newcommand{\comm}[1]{}
\newcommand{\pres}[2]{\langle #1 \mid #2 \rangle} 
\newcommand{\ol}[1]{\overline{#1}}
\newcommand{\gen}[1]{\ol{\langle #1 \rangle}}
\newcommand{\D}{\Delta}

\makeatletter
\def\moverlay{\mathpalette\mov@rlay}
\def\mov@rlay#1#2{\leavevmode\vtop{%
   \baselineskip\z@skip \lineskiplimit-\maxdimen
   \ialign{\hfil$\m@th#1##$\hfil\cr#2\crcr}}}
\newcommand{\charfusion}[3][\mathord]{
    #1{\ifx#1\mathop\vphantom{#2}\fi
        \mathpalette\mov@rlay{#2\cr#3}
      }
    \ifx#1\mathop\expandafter\displaylimits\fi}
\makeatother

\newcommand{\bigcupdot}{\charfusion[\mathop]{\bigcup}{\cdot}}

\begin{document}

\title {On Pro-$p$ groups with quadratic cohomology} 

\author{C. Quadrelli \and I. Snopce \and M. Vannacci}

\date{\today}
\address{Dipartmento di Matematica e Applicazioni, Universit\`a di Milano-Bicocca \\ 
 Milan, Italy}
\email{claudio.quadrelli@unimib.it}
\address{Instituto de Matem\'atica, Universidade Federal do Rio de Janeiro \\
 Rio de Janeiro, Brasil}
\email{ilir@im.ufrj.br}
\address{Matematika Saila, Universidad del Pa\'is Vasco \\
Bilbao, Spain}
\email{matteo.vannacci@ehu.eus}

\begin{abstract}
The main purpose of this article is to study pro-$p$ groups with quadratic $\F_p$-cohomology algebra, i.e.\ $H^\bullet$-quadratic pro-$p$ groups. Prime examples of such groups are the maximal Galois pro-$p$ groups of fields containing a primitive root of unity of order $p$. 

We show that the amalgamated free product and HNN-extension of $H^\bullet$-quadratic pro-$p$ groups is $H^\bullet$-quadratic, under certain necessary conditions. Moreover, we introduce and investigate a new family of pro-$p$ groups that yields many new examples of $H^\bullet$-quadratic groups: $p$-RAAGs. These examples generalise right angled Artin groups in the category of pro-$p$ groups. Finally, we explore ``Tits alternative behaviour'' of $H^\bullet$-quadratic pro-$p$ groups.
\end{abstract}



\thanks{The first author was partially supported by the grant ``Giovani Talenti 2017''.
The second author acknowledges support from the Alexander von Humboldt Foundation, CAPES (grant 88881.145624/2017-01) and FAPERJ.  The third author acknowledges support from the research training group
\textit{GRK 2240: Algebro-Geometric Methods in Algebra, Arithmetic and Topology}, funded by the DFG}

\maketitle

\section{Introduction}
\subsection{Number-theoretic motivation} For a field $K$, let $\bar K_s$ denote its separable closure.
The absolute Galois group of $K$ is the profinite group $G_K=\Gal(\bar K_s/K)$.
One of the main challenges in current Galois theory is to describe absolute Galois groups of fields among profinite groups.
Already describing the maximal pro-$p$ quotient $G_K(p)$ of $G_K$ among pro-$p$ groups, for $p$ a prime number,
is a remarkable challenge.

The most amazing advancement in Galois theory in the last decades is the proof
of the \emph{Bloch-Kato Conjecture} by M.~Rost and V.~Voevodsky, with the contribution of Ch.~Weibel
(cf.\ \cite{rost:BK, voev, weibel}). One of the most important consequences of this result is the following:
if $K$ contains a root of unity of order $p$, then the $\F_p$-cohomology algebra of $G_K$ ---
i.e.\ the graded algebra $\bigoplus_{n\geq0}H^n(G_K,\F_p)$, endowed with the cup product and with $\F_p$ as a trivial $G_K$-module ---
is a \emph{quadratic algebra} over $\F_p$, namely, all its elements of positive degree are combinations
of products of elements of degree 1, and its defining relations are homogeneous relations of degree 2
(see Definition~\ref{def:quadratic algebra}).
This property is inherited by the maximal pro-$p$ quotient $G_K(p)$ (cf.\ \cite[\S~2]{cq:bk}).

It is therefore of major interest to study pro-$p$ groups whose $\F_p$-cohomology algebra is quadratic, 
which we call \emph{$H^\bullet$-quadratic} (or simply \emph{quadratic}) pro-$p$ groups.
This class has been investigated in some recent papers, in order to find obstructions in the realization
of pro-$p$ groups as the maximal pro-$p$ quotient of absolute Galois groups (cf.\ \cite{cq:bk,cq:qc,qw:cyclotomic}).
Unfortunately, as it often happens in profinite group theory, there is an astounding lack of examples of $H^\bullet$-quadratic pro-$p$ groups. We will try to partly remedy this lack of examples in our work.

\subsection{Main Results} In the present paper we start a systematic investigation of $H^\bullet$-quadratic pro-$p$ groups from various points of view. 

First of all, we deal with torsion in $H^\bullet$-quadratic pro-$p$ groups. In the next proposition we will collect some facts that might be already known to experts. For further discussion on finite $H^\bullet$-quadratic $2$-groups see Remark~\ref{rmk:propCp2}.

\begin{proABC}\label{prop:properties quadratic}
Let $G$ be a finitely generated pro-$p$ group.
 \begin{itemize}
   \item[(a)] If $p$ is odd and $G$ is $H^\bullet$-quadratic, then $G$ is torsion-free.
  \item[(b)] If $p=2$, $G$ is finite and abelian, then $G$ is $H^\bullet$-quadratic if and only if $G$ is $2$-elementary abelian.
  \item[(c)]  If $p=2$ and $G$ is finite, then every subgroup of $G$ is quadratic if and only if $G$ is $2$-elementary abelian.
  \end{itemize}
\end{proABC}

Next, we study the closure of the class of $H^\bullet$-quadratic pro-$p$ groups under free constructions in the category of pro-$p$ groups, such as
\emph{amalgamated free products} and \emph{HNN-extensions} (see Section~\ref{sec:combgrth} for the definitions).
The corresponding results for free products and direct products are well known to experts. In Section~\ref{sec:combgrth} we prove the following theorems.

\begin{thmABC}\label{thm:cohomology amalgam}
Let $G$ be a finitely generated pro-$p$ group which can be written as a proper amalgam $G=G_1\amalg_{H}G_2$ with $H^\bullet$-quadratic
pro-$p$ groups $G_1,G_2,H$.
Assume that the restriction maps
\begin{equation}\label{eq:surj cohomology amalgam}
  \mathrm{res}_{G_i,H}^\bullet\colon H^\bullet(G_i,\F_p)\longrightarrow H^\bullet(H,\F_p),
\end{equation}
with $i=1,2$, satisfy the following conditions:
\begin{itemize}
 \item[(i)] $\res_{G_1,H}^1$ and $\res_{G_2,H}^1$ are surjective;
 \item[(ii)] $\kernel(\res_{G_i,H}^2)=\kernel(\res_{G_i,H}^1)\wedge H^1(G_i,\F_p)$ for both $i=1,2$.
\end{itemize}
Then also $G$ is $H^\bullet$-quadratic. 
\end{thmABC}

\begin{thmABC}\label{thm:HNN}
 Let $G_0$ be a $H^\bullet$-quadratic pro-$p$ group, and let $A,B\le G_0$ two isomorphic
$H^\bullet$-quadratic subgroups, with isomorphism $\phi\colon A\to B$.
Assume that 
\begin{itemize}
 \item[(i)] the restriction map $\res_{G_0,A}^1$ is surjective,
\item[(ii)] $\kernel(\res_{G_0,A}^2)=\kernel(\res_{G_0,A}^1)\wedge H^1(G_0,\F_p)$ and
 \item[(iii)] the map $f^1_{G_0}\colon H^1(G_0,\F_p)\to H^1(A,\F_p)$ is trivial, where
$$f_{G_0}^1=\res_{G_0,A}^1-\phi^*\circ\res_{G_0,B}^1,$$
and $\phi^*\colon H^1(B,\F_p)\to H^1(A,\F_p)$ is the map induced by $\phi$.
\end{itemize}
Assume further that $G=\mathrm{HNN}(G_0,A,\phi)$ is proper.
Then $G$ is $H^\bullet$-quadratic.
\end{thmABC}

Moreover, we exhibit several examples to show that each numbered condition in Theorem~\ref{thm:cohomology amalgam}
and Theorem~\ref{thm:HNN} is necessary (see Examples~\ref{example:amalg1}, \ref{example:amalg2}, \ref{ex:HNNcondi}, \ref{ex:HNNcondii} and \ref{ex:HNNcondiii}). In passing we find two new criteria to insure that an amalgam of pro-$p$ groups is \emph{proper}, see Proposition~\ref{prop:amalgunif} and Proposition~\ref{lem:amalgunifsubgroup}. In particular, in Proposition~\ref{prop:amalgunif} we show that the amalgamated free product of two uniform groups $G_1$ and $G_2$ over an isomorphic uniform subgroup $H$ is always proper, provided that the generators of $H$ are part of bases of both $G_1$ and $G_2$.

\begin{tcolorbox}
In the rest of this section $p$ will be a prime number greater or equal to $3$. For $p=2$ see Section~\ref{sec:p=2}.
\end{tcolorbox}

We proceed by collecting some known results on $p$-adic analytic pro-$p$ groups to characterise $H^\bullet$-quadratic groups in this class (see Section~\ref{sec:quadraticpadic} for the definitions).

\begin{thmABC}\label{thm:quadratic analytic}
 A $p$-adic analytic pro-$p$ group $G$ is $H^\bullet$-quadratic if and only if $G$ is uniform.
\end{thmABC}

Note that this provides a new characterisation of uniform pro-$p$ groups among $p$-adic analytic ones. As a direct consequence we have that if a $p$-adic analytic pro-$p$ group $G$ can be realised as  $G_K(p)$ for some $K$ containing a root of unity of order $p$ then every open subgroup of $G$ is uniform, i.e., $G$ is a hereditarily uniform pro-$p$ group. Note that  hereditarily uniform pro-$p$ groups have been classified and all those groups can be realised as maximal  pro-$p$ Galois groups (cf. \cite{cq:bk}, \cite{ware} and \cite{ks}; see also \cite{ks-j.algebra}).

We continue by looking at $H^\bullet$-quadratic groups from a more ``combinatorial'' point of view.
Namely we introduce a new class of pro-$p$ groups, called \emph{generalised Right Angled Artin pro-$p$ Groups}
--- \emph{$p$-RAAGs} for short.
Define a \emph{$p$-graph} $\Gamma=(\calG,f)$ to be an oriented graph $\calG = (\sV,\sE)$ together with labels $f(e)=(f_1(e),f_2(e))\in p\Z_p\times p\Z_p$
for every edge $e\in\sE$.
The $p$-RAAG associated to $\Gamma$ is the pro-$p$ group defined by the pro-$p$ presentation 
\begin{equation*}
 G_\Gamma = \pres{x_1,\ldots,x_n}{[x_i,x_j]= x_i^{f_1(e)} x_j^{f_2(e)}\ \text{for } e=(x_i,x_j)\in \sE} 
\end{equation*}
(see Definition~\ref{def:pRAAGs} for the precise details). The class of $p$-RAAGs is extremely interesting.
For instance, we can completely  determine the $\F_p$-cohomology algebra of an $H^\bullet$-quadratic $p$-RAAG
and this only depends on the underlying graph.

\begin{thmABC}\label{thm:pRAAGs quadratic}
Let $G_\Gamma$ be an $H^\bullet$-quadratic $p$-RAAG with associated $p$-graph $\Gamma=(\mathcal{G},f)$.
Then  
\begin{equation*}
 H^\bullet(G_\Gamma,\F_p)\cong \Lambda_\bullet(\mathcal{G}^{\mathrm{op}}).
\end{equation*}
\end{thmABC}

In Theorem~\ref{thm:galois pRAAGs} we show that there is a further condition, besides yielding
a $H^\bullet$-quadratic pro-$p$ RAAG, that a $p$-graph $\Gamma$ must satisfy in order to generate 
a $p$-RAAG which occurs as a Galois group $G_K(p)$ for some field $K$.

Next we prove that $p$-RAAGs provide an abundance of new examples of $H^\bullet$-quadratic pro-$p$ groups.

\begin{thmABC}\label{thm:pRAAGs mild}
 Let $\Gamma=(\mathcal{G},f)$ be a $p$-graph.
Then the following are equivalent.
\begin{itemize}
 \item[(i)] $\mathcal{G}$ is triangle-free.
 \item[(ii)] The generalised $p$-RAAG $G_\Gamma$ is mild.
 \item[(iii)] $G_\Gamma$ has cohomological dimension $\cdd(G_\Gamma)=2$.
\end{itemize}
In particular, if these conditions are satisfied, $G_\Gamma$ is quadratic.
\end{thmABC}

In fact, by a classical theorem of Erd\"os et al., the number of graphs on $n$ vertices that do not involve a
``triangle'' is asymptotic to $2^{n^2/4}$ for $n$ that tends to infinity;
while the total number of graphs on $n$ vertices is asymptotic to $2^{n^2/2}$ for $n$ that tends to infinity.
The condition of \emph{mildness} for a pro-$p$ group, introduced by J.~Labute in \cite{labute:mild},
is quite technical and we direct the reader to \S~\ref{sec:mild} and references therein for the definition.

In light of Theorem~\ref{thm:pRAAGs mild}, we start a careful investigation of $p$-RAAGs associated
to triangle $p$-graphs. It so happens that a classical example of Mennicke, of a \emph{finite} $3$-generated $3$-related pro-$p$ group, can be written as a triangle $p$-RAAG (cfr.\ Example~\ref{ex:mennicke}); so there are $p$-RAAGs that are not quadratic, because they contain non-trivial torsion. Furthermore, it is in general very hard to decide whether a given presentation of a pro-$p$ group yields a finite group. So we content ourselves to classify the possible \emph{$H^\bullet$-quadratic} $p$-RAAGs that arise from triangle $p$-RAAGs.

\begin{thmABC}\label{thm:quadratictrianglepraags}
Let $G$ be a $p$-RAAG associated to a triangle $p$-graph, i.e., let $$G = \pres{x,y,z}{[x,y]=x^{\a_1} y^{\a_2},\ [y,z]=y^{\b_2} z^{\b_3},\ [z,x]=z^{\c_3} x^{\c_1}},$$ where $\a_1,\a_2, \b_2,\b_3, \c_1,\c_3 \in p\mathbb{Z}_p$. If $G$ is
 $H^\bullet$-quadratic,  then either:
 \begin{itemize}
     \item[(a)] $G$ is a metabelian uniform pro-$p$ group; or 
     \item[(b)] $G$ is isomorphic to a $3$-dimensional uniform open subgroup of the $p$-Sylow subgroup of $\mathrm{SL}_2(\mathbb{Z}_p)$.
 \end{itemize}
\end{thmABC}

Additionally, if $G$ is metabelian in part (a) of the previous theorem,
it has to belong to one of three different explicitly described families of pro-$p$ groups. 
We investigate also part (b) in more detail: namely, we explicitly realise $\mathrm{SL}_2^1(\mathbb{Z}_p)$
as a triangle $p$-RAAG (see Proposition~\ref{prop:SL21triangle}).

To boot, we analyse whether all \emph{torsion-free} $p$-RAAGs are $H^\bullet$-quadratic.
We fall short of proving this in full generality, but we can use Theorem~\ref{thm:cohomology amalgam}
to deal with a large number of them (see \S~\ref{sec:sometrianglefull}). In particular, we can prove the following theorem for
$p$-RAAGs associated to \emph{chordal} graphs. A graph is chordal if all cycles of four or more vertices have a chord (see Definition~\ref{def:chordal}).

\begin{thmABC}\label{thm:chordal}
Let $\Gamma=(\calG,f)$ be a $p$-graph with $\calG$ a chordal graph, such that the associated $p$-RAAG $G_\Gamma$ is non-degenerate.
Then $G$ is a quadratic pro-$p$ group.
\end{thmABC}

Morever, we can prove that all $p$-RAAGs arising from $p$-graphs on $5$ vertices are $H^\bullet$-quadratic,
if they are \emph{non-degenerate} (cf.\ Definition~\ref{defn:nondeg}).

Finally, we investigate the presence of free non-abelian closed subgroups in $H^\bullet$-quadratic pro-$p$ groups.
In the arithmetic case one has a ``Tits alternative type'' result: if
the field $K$ contains a roots of unity of order $p$, then either $G_K(p)$ contains a free non-abelian 
closed pro-$p$ subgroup, or every finitely generated subgroup is uniform
(cf. \cite[Thm.~B]{cq:bk}, \cite[Thm.~3]{ware} and \cite[\S~3.1]{CMQ:fast}).
We prove the following (see \S~\ref{sec:free subgp}).

\begin{thmABC}\label{thm:pRAAGs free subgroup}
Let $G_\Gamma$ be a $p$-RAAG, with associated $p$-graph $\Gamma=(\calG,f)$. 
Then either $G_{\Gamma}$ is a powerful pro-$p$ group, or it contains a free non-abelian closed pro-$p$ subgroup. 
\end{thmABC}

As an immediate corollary of the previous theorem we deduce: a quadratic $p$-RAAG is either uniform, or it contains a free non-abelian closed pro-$p$ subgroup.

We also investigate the presence of free non-abelian subgroups in mild pro-$p$ groups. For this, in Proposition~\ref{prop:mild GoSha} we show that ---under certain conditions--- several mild pro-$p$ groups are generalised Golod-Shafarevic pro-$p$ groups (see Section~\ref{sec:ggs} for the definition). As a corollary, we deduce that all $H^\bullet$-quadratic groups with at most $3$ generators are either uniform or contain a closed free non-abelian pro-$p$ subgroup (see Corollary~\ref{coro:small free subgp}). Motivated by the aforementioned results, we formulate the following.

\begin{conjABC}\label{conj:subgp free}
Let $G$ be a finitely generated $H^\bullet$-quadratic pro-$p$ group. Then either $G$ is a uniform pro-$p$ group, or it contains a closed free non-abelian pro-$p$ subgroup.
\end{conjABC}

\subsection{Case \texorpdfstring{$\boldsymbol{p=2}$}{p=2}}\label{sec:p=2}
As it often happens in the theory of pro-$p$ groups, one needs to take extra care in the case $p=2$. We will assume a technical condition on $H^\bullet$-quadratic pro-$2$ groups (see Remark~\ref{rem:p2}) that will insure that most of our proofs also work for the even prime. 

We remark that Theorems~\ref{thm:cohomology amalgam} and \ref{thm:HNN} hold without change for $p=2$. 
 Theorem~\ref{thm:quadratic analytic} does not hold for $p=2$, see Examples~\ref{ex:analyticp21} and \ref{ex:analyticp22}. 

The definition of $p$-RAAGs needs to be slightly modified for $p=2$ (see Section~\ref{sec:pRAAGs}), namely labels of edges have to be in $4\mathbb{Z}_2$. After this modification, Theorems~\ref{thm:pRAAGs quadratic}, \ref{thm:pRAAGs mild}, \ref{thm:quadratictrianglepraags}, \ref{thm:chordal} and \ref{thm:pRAAGs free subgroup} hold for $p=2$.

Finally, Proposition~\ref{prop:mild GoSha} also holds for $p=2$ under the additional assumption of Remark~\ref{rem:p2}.

\subsection{Structure of the article}
The main definitions and conventions of this article are laid out in Section~\ref{sec:Preliminaries}. We prove Proposition~\ref{prop:properties quadratic} in Section~\ref{sec:quadratic}. Theorems~\ref{thm:cohomology amalgam} and \ref{thm:HNN} are proved in Section~\ref{sec:combgrth}, where we also discuss the properness of amalgamated free products and HNN extensions in the category of pro-$p$ groups. Theorem~\ref{thm:quadratic analytic} is proved in Section~\ref{sec:quadraticpadic}. We introduce the class of $p$-RAAGs in Section~\ref{sec:pRAAGs}, where we also prove Theorems~\ref{thm:pRAAGs quadratic} and \ref{thm:pRAAGs mild}. Non-degenerate $p$-RAAGs are defined and studied in Section~\ref{sec:trianglefull}. Triangle $p$-RAAGs are analysed in Section~\ref{sec:triang}, which also includes a proof of Theorem~\ref{thm:quadratictrianglepraags}. Finally, in Section~\ref{sec:free subgp} we investigate ``Tits alternative behaviour'' in the classes of $p$-RAAGs and mild pro-$p$ groups, proving Theorem~\ref{thm:pRAAGs free subgroup}.

\section{Preliminaries}\label{sec:Preliminaries}

\subsection{Quadratic algebras}

An associative unital algebra $A_\bullet$ over the finite field $\F_p$ is graded if it decomposes
as the direct sum of $\F_p$-vector
spaces $A_\bullet=\bigoplus_{n\ge 0}A_n$ such that $A_n\cdot A_m\subseteq A_{n+m}$ for all $n,m\ge 0$.
The graded algebra $A_\bullet$ is of finite type if every space
$A_n$ has finite dimension. 
Hereinafter we will restrict ourself to  
graded $\F_p$-algebras of finite type,
and to finite dimensional vector spaces over $\F_p$. 

For an $\F_p$-vector space $V$, let $T_\bullet(V)=\bigoplus_{n\geq0}V^{\otimes n}$ denote the free graded tensor algebra generated by $V$. For $\Omega\subseteq V\otimes V$, let $(\Omega)$ denote the two-sided ideal generated by $\Omega$ in $T_\bullet(V)$. 

\begin{definition}\label{def:quadratic algebra}
 A graded algebra $A_\bullet$ is called \emph{quadratic} if one has an isomorphism
 \begin{equation}\label{eq:quadratic algebra}
 A_\bullet\cong   Q(A_1,\Omega) := \frac{T_\bullet(A_1)}{(\Omega)},
 \end{equation}
for some $\Omega\subseteq A_1\otimes A_1$. We will call $Q(A_1,\Omega)$ the \emph{quadratic algebra} generated by $A_1$ and $\Omega$.
\end{definition}

\begin{example}\label{ex:quad algebras}
Let $V$ be a vector space. Then 
\begin{itemize}
  \item[(a)] the tensor algebra $T_\bullet(V) = Q(V,0)$,
  \item[(b)] the trivial algebra $ Q(V,V^{\otimes 2})$, 
  \item[(c)] the symmetric algebra $S_\bullet(V) = Q(V,\{v\otimes w - w\otimes v\mid v,w\in V\})$ and 
  \item[(d)] the exterior algebra $\Lambda_\bullet(V) = Q(V, \{v\otimes v\mid v\in V\})$
  \end{itemize}
  are quadratic.  Moreover, for two quadratic algebras $A_\bullet,B_\bullet$, one has the following general constructions of new quadratic algebras (cf.~\cite[\S~3.1]{pp:quad}). 
 \begin{itemize}
  \item[(a)] The direct sum $C_\bullet=A_\bullet\oplus B_\bullet$ is the algebra with
  $C_n=A_n\oplus B_n$ for every $n\geq1$.
  \item[(b)] The wedge product $C_\bullet=A_\bullet\wedge B_\bullet$ is the algebra with 
  $C_n=\bigoplus_{i+j=n}A_i\wedge B_j$ for every $n\geq2$ (note that in \cite{pp:quad,MPQT} this algebra is
  called the skew-commutative tensor product $A_\bullet\otimes^{-1}B_\bullet$).
 \end{itemize}
\end{example}

\subsection{Presentations of pro-\texorpdfstring{$p$}{p} groups}
Throughout this paper, subgroups of pro-$p$ groups are assumed to be closed (in the pro-$p$ topology) and generators will be intended as topological generators. In particular, given two (closed) subgroups $H_1$ and $H_2$ of a pro-$p$ group $G$, the subgroup $[H_1,H_2]$ is the (closed) subgroup of G generated by the commutators $[g_1,g_2]$, with $g_i \in H_i$ for $i = 1,2$. Also, for a positive integer $n$, $G^n$ denotes the (closed) subgroup of $G$ generated by the $n$-th powers of elements of $G$. Similarly, all homomorphisms between pro-$p$ groups will be assumed continuous. 
Finally, we will always consider $\F_p$ as a trivial $G$-module.

We will now recall some basic facts and definitions for presentations of pro-$p$ groups. The experienced reader might wish to skip ahead to the next section. 

For a pro-$p$ group $G$, set $G_{(2)}=G^p [G,G]$ and 
\begin{equation} \label{eq:G3}
 G_{(3)}=\begin{cases} G^p [[G,G],G] & \text{if }p\geq3 \\ G^4 [[G,G],G] & \text{if }p=2 \end{cases}.
\end{equation}
Then $G_{(2)}$ and $G_{(3)}$ coincide with the second and third terms of the $p$-Zassenhaus filtration of $G$,
respectively (cf.\ \cite[\S~11.1]{ddms:padic}).
The pro-$p$ group $G$ is finitely generated if and only if $G/G_{(2)}$ is a vector space of finite dimension over $\F_p$, and 
$\dd(G):=\dim G/G_{(2)}$ is the minimal number of generators of $G$.
Moreover, the equality $H^1(G,\F_p)=\mathrm{Hom}(G,\F_p)$ implies the isomorphism
of vector spaces 
\begin{equation}\label{eq:H1}
 H^1(G,\F_p)\cong(G/G_{(2)})^*,
\end{equation}
where $\cdot^*=\mathrm{Hom(\cdot,\F_p)}$ denotes the dual for vector spaces.
Thus, $\dd(G)=\dim H^1(G,\F_p)$.

A presentation
\begin{equation}\label{eq:presentation}
 \xymatrix{ 1\ar[r] & R\ar[r] & F\ar[r] & G\ar[r] & 1 }
\end{equation}
of a finitely presented pro-$p$ group $G$ is said to be \emph{minimal} if one of the following equivalent conditions is satisfied:
\begin{itemize}
 \item[(i)] $F\to G$ induces an isomorphism $F/F_{(2)}\cong G/G_{(2)}$;
 \item[(ii)] $R\le F_{(2)}$.
 \end{itemize}
 If \eqref{eq:presentation} is minimal, it follows from \cite[\S~1.6]{nsw:cohn} that 
$$H^2(G,\F_p)\cong H^1(R,\F_p)^F\cong  (R/R^p[R,F])^*.$$
A minimal subset $\mathcal{R}\subseteq F$ which generates $R$ as normal subgroup is called a set of defining relations of $G$.
Thus, $\rr(G):=\dim H^2(G,\F_p)$ is the cardinality of a set of defining relations.

For a pro-$p$ group $G$ and a minimal presentation \eqref{eq:presentation}, set $d=\dd(G)$ and let $\mathcal{X}=\{x_1,\ldots,x_d\}$ be a basis of $F$. We will identify $\mathcal{X}$ with the induced basis of $G$ via $F\to G$.
Given $r\in R \subseteq F_{(2)}$, one may write
\begin{equation}\label{eq:shape r}
 r=\begin{cases}
    \prod_{i<j}[x_i,x_j]^{a_{ij}}\cdot r' & \text{if }p\neq2 \\ 
    \prod_{i=1}^d x_i^{2a_{ii}}\cdot\prod_{i<j}[x_i,x_j]^{a_{ij}}\cdot r' & \text{if }p=2 
   \end{cases} \qquad r'\in F_{(3)},
\end{equation}
with $0\leq a_{ij}<p$, and such numbers are uniquely determined by $r$
(cf.\ \cite[Prop.~1.3.2]{vogel:thesis}).

\subsection{Pro-\texorpdfstring{$p$}{p} groups and \texorpdfstring{$\mathbb{F}_p$}{Fp}-cohomology}

The $\mathbb{F}_p$-cohomology of a pro-$p$ group $G$ comes endowed with the cup-product
\[
 \smile: H^i(G,\F_p)\times H^j(G,\F_p)\longrightarrow H^{i+j}(G,\F_p)
\]
which is bilinear and graded-commutative, i.e.\ $\beta\smile \alpha=(-1)^{ij}\alpha\smile \beta$ for $\alpha\in H^i(G,\F_p)$
and $\beta\in H^j(G,\F_p)$, so that $H^\bullet(G,\F_p)=\bigoplus_{n\geq0}H^n(G,\F_p)$
is a graded algebra (cf.\ \cite[\S~I.4]{nsw:cohn}).

Let $G$ be a finitely presented pro-$p$ group with minimal presentation \eqref{eq:presentation}, with
$\mathcal{X}$ as above, and let $\mathcal{X}^*=\{\alpha_1,\ldots,\alpha_d\}$ be the basis of $H^1(G,\F_p)$
dual to $\mathcal{X}$.
Moreover, set $m=\rr(G)$ and let $\mathcal{R}=\{r_1,\ldots, r_m\}$ be a set of defining relations of $G$.
Cup-products of elements of degree 1 and defining relations are connected by the following
(cf.\ \cite[Prop.~1.3.2]{vogel:thesis} and \cite[Prop.~7.1]{MPQT}).

\begin{proposition}\label{prop:cupproduct}
 Let $G$, $\mathcal{X}$, $\mathcal{X}^*$ and $\mathcal{R}$ be as above.
 Then for every $r_h\in\mathcal{R}$ one has a bilinear pairing 
 \[
  \mathrm{trg}^{-1}(\alpha_i\smile \alpha_j).r_h=\begin{cases}
    -a_{ij} & \text{if }i<j \\ a_{ij} & \text{if }j<i \\  -\binom{p}{2}a_{ii} & \text{if }i=j,                                                                             
                                                                                \end{cases} \]
where $a_{ij}$ are the exponents in the expression of $r_h$ as in \eqref{eq:shape r},
interpreted as elements of $\F_p$.
In particular, $\alpha_i \smile\alpha_j\neq0$ if and only if $a_{ij}\neq0$ for some $r_h\in\mathcal{R}$.
\end{proposition}

Moreover, one has the following (cf.\ \cite[Thm.~7.3]{MPQT}).

\begin{proposition}\label{prop:relations}
 Let $G$ be a finitely generated pro-$p$ group.
 Then the following are equivalent.
 \begin{itemize}
  \item[(i)] The cup-product induces an epimorphism $H^1(G,\F_p)^{\otimes2}\to H^2(G,\F_p)$.
  \item[(ii)] One has an equality $R\cap F_{(3)}=R^p[R,F]$  
  for every $i$.
  \end{itemize}
If the above two conditions hold, then $r_i\in R\smallsetminus F_{(3)}$ for any set of defining relations $\mathcal{R}=\{r_1,\ldots,r_m\}$ of $G$.
\end{proposition}

Hence, if a pro-$p$ group $G$ can have a chance to be quadratic (see Definition~\ref{definition:Hquadratic}), it has to satisfy one of the equivalent conditions of Proposition~\ref{prop:relations}. In the rest of the paper we will mostly restrict ourselves to this situation. In fact, one of the easiest shapes of relations satisfying (iii) above is involved in the definition of $p$-RAAGs (see Section~\ref{eq:pRAAGdefn}).

\begin{remark}[Gauss reduction of relations]\label{rem:Gauss reduction}
 Let $G$ and $\mathcal{X}$ be as above and set $d=\dd(G)$, $m=\rr(G)$. We will consider the lexicographic order $<$ on the set of couples $\{(i,j) \mid 1\le i<j\le m\}$.
The quotient $F_{(2)}/F_{(3)}$ is an $\F_p$-vector space with basis $\{[x_i,x_j]F_{(3)} \mid 1\leq i<j\leq d\}$.
Suppose that $G$ satisfies the conditions of Proposition~2.4 and let $\mathcal{R}=\{r_1,\ldots,r_m\}$ be a set of defining relations of $G$. Then $R/R\cap F_{(3)}$ is a subspace with basis $\bar{\mathcal{R}}=\{r_h F_{(3)} \mid 1\leq h\leq m\}$ and we will write $$r_h F_{(3)} = \sum_{1\le i< j\le d} a(h,i,j) \cdot [x_i,x_j]F_{(3)}$$ with $a(h,i,j) \in \mathbb{F}_p$.
 Let $A=(a(h,i,j))_{h,(i,j)}$ be the $m\times {d \choose 2}$ matrix of the coefficients of the elements of $\mathcal{R}$ in $F_{(2)}/F_{(3)}$. 
Gauss reduction on $A$ yields a lexicographically ordered sequence $(i_1,j_1)<\ldots<(i_m,j_m)$ 
and a set $\mathcal{R}' = \{r_1',\ldots,r_m'\}\subseteq F$ of new relators for $G$ such that:
 $$r_h' \equiv \prod_{1\leq i<j\leq d}[x_i,x_j]^{b(h,i,j)}\mod F_{(3)},\quad h=1,\ldots,m$$ with exponents $b(h,i,j)\in\F_p$ such that $$b(h,i,j) = \begin{cases} 
 1 & \text{ if } (i,j) = (i_h,j_h) \\  
 0 & \text{ if } (i,j)<(i_h,j_h)    
 \end{cases}$$
The commutator $[x_{i_h},x_{j_h}]$ may be considered the ``leading term'' of $r_h$. This will be relevant in Section~\ref{sec:free subgp}.
\end{remark}

\subsection{\texorpdfstring{$H^\bullet$}{H}-quadratic pro-\texorpdfstring{$p$}{p} groups}\label{sec:quadratic}

We introduce the main object of our investigations.

\begin{definition}\label{definition:Hquadratic}
 A pro-$p$ group is called \emph{$H^\bullet$-quadratic} (or simply \emph{quadratic}) if the $\F_p$-cohomology algebra $H^\bullet(G,\F_p)$,
 endowed with the cup-product, is a quadratic algebra.
\end{definition}

In the rest of the paper we will refer to $H^\bullet$-quadratic pro-$p$ groups simply as \emph{quadratic pro-$p$ groups}. 

Let $G$ be a quadratic pro-$p$ group.
If $p\neq2$, then $\alpha\smile \alpha=0$ for every $\alpha\in H^\bullet(G,\F_p)$ by graded-commutativity.
If we further assume that $\alpha\smile \alpha=0$ for every $\alpha$ also in the case $p=2$,
then one has an epimorphism of quadratic algebras
\begin{equation}\label{eq:wedge}
\end{equation}
Recall that $\cdd(G)$ denotes the cohomological dimension of $G$ (cf.\ \cite[\S~III.3]{nsw:cohn}). In particular, if $G$ is finitely generated then one has the inequalities
\begin{equation}\label{eq:inequalities}
 \cdd(G)\leq \dd(G) \quad \text{and} \quad \rr(G)\leq\binom{\dd(G)}{2},
\end{equation}
so that $G$ is finitely presented.

\begin{notation}
With a slight abuse of notation, from now on we denote the cup product $\a \smile \b$ of $\a \in H^i(G,\mathbb{F}_p)$ and $\b \in H^j(G,\mathbb{F}_p)$ by $\a\wedge \b  \in H^{i+j}(G,\mathbb{F}_p)$. Moreover, when it is clear from context, the cup-product symbol will be omitted.
\end{notation}

\begin{example}
Free pro-$p$ groups are trivially quadratic, as $\cdd(G)=1$ for a free pro-$p$ group $G$ (cf.\ \cite[Prop.~3.5.17]{nsw:cohn}), i.e.\ $H^n(G,\F_p)=0$ for $n\geq2$.
\end{example}

Given a field $K$, let $G_K(p)$ denote the maximal pro-$p$ quotient of the absolute Galois group $G_K$ --- 
i.e.\ $G_K(p)$ is the Galois group of the maximal $p$-extension of $K$.
Then one has the following consequence of the Rost-Voevodsky Theorem (cf.\ \cite[\S~24.3]{ido:miln}),
which is one of the reasons of our interest in quadratic pro-$p$ groups.

\begin{theorem}\label{thm:GKp quadratic}
  Let $K$ be a field containing a root of 1 of order $p$.
Then the maximal pro-$p$ Galois group $G_{K}(p)$ of $K$ is quadratic.
\end{theorem}

The well-known Artin-Schreier Theorem (and its pro-$p$ version) implies that the only finite group
which occurs as maximal pro-$p$ Galois group is $C_2$, the cyclic group of order~$2$.
We will now prove Proposition~\ref{prop:properties quadratic}.

\begin{proof}[\textbf{Proof of Proposition~\ref{prop:properties quadratic}}]
For (a), assume by contradiction that $G$ is not torsion-free.
 Then $G$ contains a subgroup $H$ which is cyclic of order $p$, so that $H^n(H,\F_p)\cong \F_p$
 for every $n\geq0$ and $\cdd(H)$ is infinite.
 Thus, also $\cdd(G)$ is infinite by \cite[Prop.~3.3.5]{nsw:cohn}, contradicting \eqref{eq:inequalities}. This proves claim (a).
%

Regarding (b), it is well known that the cohomology algebra of a $2$-elementary abelian group $C_2^d$ ($d\in \mathbb{N}$) with coefficients in the finite field $\F_2$ is the symmetric algebra $S_\bullet(H^1(C_2^d,\F_2))$. Thus, $C_2^d$ is quadratic. 

Now, let $G$ be a finite $2$-group and suppose that $G$ has a minimal presentation of the form $\langle X \vert R \rangle$ with $x,y \in X$ and either $x^{2^k}y^{-2^m}\in R$ or $x^{2^k}\in R$ for some $m,k\ge 2$. Then $G$ is not quadratic. Since the only finite abelian $2$-groups that do not have relations of the above form are elementary abelian, item (b) follows.

For (c), by item (b), all elements of $G$ must have order $2$ and hence $G$ is abelian.
\end{proof}

We were informed by J.\ Min\'a\v{c} that a similar result to our Proposition~\ref{prop:properties quadratic} has been obtained also by him together with D.\ Benson, S.\ Chebolu, C.\ Okay and J.\ Swallow, and it will appear in a forthcoming paper.

\begin{remark}[Proposition~\ref{prop:properties quadratic} for $p=2$]\label{rmk:propCp2}
We believe that, if a finite $2$-group is quadratic, then it must be elementary abelian. 
One can check this directly for several finite $2$-groups: dihedral groups $D_{2^n}$ (see \cite[Chap.~IV, Thm.~2.7]{atem}), quaternion group and generalised quaternion groups (see \cite[Chap.~IV, Thm.~2.9]{atem} and \cite[Chap.~IV, Lem.~2.11]{atem}) and extraspecial $2$-groups (see \cite[Rem.~2.13]{atem}). 

Unfortunately, we have not been able to verify our conviction in full generality, but we have two additional comments.
First, we can show---using a spectral sequence argument---that the Mennicke $2$-group $M$ of order $2^{11}$ given by the presentation $$M= \langle a, b, c \mid [a, b] = b^{-2}, [b, c] = c^{-2}, [c, a] = a^{-2} \rangle $$ is not quadratic. We have decided to not include a proof, as this is very similar to the cited examples in \cite[Chap.~IV]{atem}.

Secondly, the method of proof of item (b) in Proposition~\ref{prop:properties quadratic} cannot be applied in general, as there are $2$-groups
that do not satisfy the condition above. These include finite $2$-groups $G$ with balanced presentations (i.e.,  with $d(G) = r(G)$), which among other groups include the Mennicke $2$-group $M$, the quaternion group and the generalised quaternion groups.
\end{remark}

\begin{remark}\label{rem:p2}
 Henceforth we will always implicitely assume that $\alpha^2$ is trivial in $H^2(G,\F_2)$ for every 
 $\alpha\in H^1(G,\F_p)$ in the case $p=2$.
 By Proposition~\ref{prop:cupproduct} and \cite[Prop.~1.3.2]{vogel:thesis},
 this is equivalent
 to assuming that $a_{ii}=0$ for every $i\in\{1,\ldots,d\}$
 and for every defining relation $r_h\in\mathcal{R}$ of $G$. 
If $G=G_K(p)$ for some field $K$, then this holds if $\sqrt{-1}\in K$. 
\end{remark}

\subsection{Mild pro-\texorpdfstring{$p$}{p} groups}\label{sec:mild}
A good source of quadratic pro-$p$ groups is the class of \emph{mild} pro-$p$ groups.
Such groups were introduced by J.~P.~Labute in \cite{labute:mild} to study the Galois groups of pro-$p$
extensions of number fields with restricted ramification.
Since a quadratic pro-$p$ group needs to satisfy the conditions of Proposition~\ref{prop:relations},
we give the definition of mild pro-$p$ groups among those satisfying these conditions.
For the general definition and properties of mild pro-$p$ groups we refer to \cite{labute:mild,forre:mild,gartner:mild}.

\begin{definition}\label{def:mild}
 Let $G$ be a finitely presented pro-$p$ group with presentation \eqref{eq:presentation}
 and satisfying Proposition~\ref{prop:relations}. Consider the graded algebra $M_\bullet(G) = Q(V,\Omega)$, where $V\cong G/G_{(2)}= \langle x_1,\ldots,x_d\rangle$
and $$\Omega=\left\{\sum_{i<j}a_{ij}(x_i \otimes x_j-x_j \otimes x_i)\right\}$$ with $a_{ij}$ as in \eqref{eq:shape r}. 
The group $G$ is \emph{mild} if one has an equality of formal power series
 \begin{equation}\label{eq:series mild}
   \sum_{n\geq 0}\dim(M_n(G))\cdot T^n=\frac{1}{1-\dd(G)T+\rr(G)T^2}.
 \end{equation}
\end{definition}

The most interesting --- and useful --- feature of mild pro-$p$ groups is that they have cohomological dimension
equal to 2 (cf.\ \cite[Thm.~1.2]{labute:mild}). In fact, it is fairly easy to decide if a pro-$p$ group with cohomological dimension $2$ is quadratic.

\begin{proposition}\label{prop:mild cd2}
 Let $G$ be a finitely generated pro-$p$ group with $\cdd(G)=2$ which satisfies Proposition~\ref{prop:relations}.
 Write $H^2(G,\F_p)=H^1(G,\F_p)^{\otimes2}/\Theta$ for some subspace $\Theta$ of $H^1(G,\F_p)^{\otimes2}$.
 If \begin{equation}\label{eq:H3}
     \Theta \wedge H^1(G,\F_p)=H^1(G,\F_p)^{\otimes3},
    \end{equation}
then $G$ is quadratic.
\end{proposition}

\begin{proof}
 Since $H^n(G,\F_p)$ is trivial for $n\geq3$, one just needs to check whether $H^3(G,\F_p)=0$
 follows from the relations $\Theta$ in $H^2(G,\F_p)$. But this follows immediately from \eqref{eq:H3}.
\end{proof}

This fact has the following consequence.

\begin{remark}
In particular, Proposition~\ref{prop:mild cd2} holds for mild groups satisfying Proposition~\ref{prop:relations}.
\end{remark}

Usually, it is quite difficult to check whether a finitely presented pro-$p$ group is mild
using the definition.
Rather, one has the following handy criterion to check whether a pro-$p$ group is mild. 

\begin{proposition}[{\cite[p.~789]{gartner:mild}}]\label{prop:cd2 cupproduct}
 Let $G$ be a finitely presented pro-$p$ group such that $H^2(G,\F_p)\neq0$.
Assume that $H^1(G,\F_p)$ admits a decomposition $H^1(G,\F_p)=V_1\oplus V_2$ such that the following holds:
\begin{itemize}
 \item[(i)] the cup-product $V_1\otimes V_1\to H^2(G,\F_p)$ is trivial;
 \item[(ii)] the cup-product $V_1\otimes V_2\to H^2(G,\F_p)$ is surjective.
\end{itemize}
Then $G$ is mild.
\end{proposition}

We remark that the condition of mildness of a pro-$p$ group satisfying Proposition~\ref{prop:relations}
depends only on the ``shape'' of its defining relations modulo $F_{(3)}$.

\begin{remark}\label{rem:mild approximation}
Let $G$ be a mild quadratic pro-$p$ group with minimal presentation \eqref{eq:presentation}
and defining relations $\mathcal{R}=\{r_1,\ldots,r_m\}$. If $\tilde{\mathcal{R}}=\{\tilde r_1,\ldots,\tilde r_m\}$
is a subset of $F$ with $\tilde r_h\equiv r_h\bmod F_{(3)}$ for every $h=1,\ldots,m$, then also the pro-$p$
group $\tilde G=F/\tilde R$ with defining relations $\tilde{\mathcal{R}}$ is mild and quadratic --- since $H^2(\tilde G,\F_p)\cong H^2(G,\F_p)$ by Proposition~2.3.
\end{remark}

The previous remark will be used in Section~\ref{sec:free subgp}.

\subsection{Generalised Golod-Shafarevic pro-\texorpdfstring{$p$}{p} groups}\label{sec:ggs}
A generalised Go\-lod-Sha\-fa\-re\-vich pro-$p$ group is a pro-$p$ group which satisfies
a weighted version of the celebrated Golod-Shafarevich condition.
We briefly recall the definition and we reference the reader \cite[\S~4.1]{ershov} for a deeper treatment. 

Let $\mathcal{X}=\{x_1,\ldots,x_d\}$ and $U=\{u_1,\ldots,u_d\}$ be two finite sets of the same cardinality.
The free pro-$p$ group $F$ on $\mathcal{X}$ can be embedded in the \emph{free associative algebra}
$\F_p \langle\!\langle U \rangle\!\rangle$  over $U$ via the Magnus embedding $\iota: x_i\mapsto 1+u_i$.
Choose weights $w(u_i)\in \mathbb{R}_{\ge0}\cup \{\infty\}$ for $x_i\in \mathcal{X}$ and extend this
to a \emph{weight function} $w: \F_p \langle\!\langle U \rangle\!\rangle \to \mathbb{R}_{\ge0}\cup \{\infty\}$
by setting 
\begin{equation*}
w(1)=0, \text{\quad } w(0)=\infty,\quad w(u_{i_1}\ldots u_{i_k})= w(u_{i_1})+ \ldots + w(u_{i_k}) 
\end{equation*} 
\begin{equation*}
w\left(\sum c_\alpha m_\alpha\right) = \min\{w(m_\alpha) \mid c_\alpha \neq 0\}.
\end{equation*} 
For an element $f\in F$, define the \emph{valuation} of $f$ by $D(f) = w(\iota(f)-1)$ and, 
for a subset $S$ of $F$, define $H_{S,D}(T) = \sum_{s\in S} T^{D(s)}$. 

A pro-$p$ group $G$ is said to be \emph{generalised Golod-Shafarevic} if there exist:
\begin{enumerate}
 \item a minimal presentation \eqref{eq:presentation} for $G$, with generators $\mathcal{X}=\{x_1,\ldots,x_d\}$ 
 and relators $\mathcal{R}$;
 \item a weight function $w: \F_p \langle\!\langle U \rangle\!\rangle \to \mathbb{R}_{\ge0}\cup \{\infty\}$ 
 (with associated valuation $D$) on $F(\mathcal{X})$ and
 \item a real number $T_0\in (0,1)$
\end{enumerate}
 such that
\[
   1- H_{\mathcal{X},D}(T_0) +   H_{\mathcal{R},D}(T_0) <0.
\]
In Section~\ref{sec:free subgp} we will use the following theorem.

\begin{theorem}[{\cite[Thm.~7.1]{ershov}}]\label{thm:ershovfreesub}
 A generalised Golod-Shafarevic pro-$p$ groups contains a free non-abelian closed subgroup. 
\end{theorem}

\section{Combinatorial group theory for quadratic pro-\texorpdfstring{$p$}{p} groups}\label{sec:combgrth}

In the next section we concern ourselves with finding new ways to construct quadratic groups from old ones.

\subsection{Free and direct products}

The following results are well-known to experts.

\begin{proposition}\label{prop:cohomology freeproduct}
 Let $G_1,G_2$ be two finitely generated quadratic pro-$p$ groups.
\begin{itemize}
 \item[(a)] The free pro-$p$ product $G=G_1\ast G_2$ is again quadratic; furthermore, we have $H^\bullet(G,\F_p)\cong H^\bullet(G_1,\F_p)\oplus H^\bullet(G_2,\F_p)$.
 \item[(b)] The direct product $G=G_1\times G_2$ is again quadratic;
 furthermore, we have $H^\bullet(G,\F_p)\cong H^\bullet(G_1,\F_p)\wedge H^\bullet(G_2,\F_p)$.
\end{itemize}
\end{proposition}

\begin{proof}
Statement (a) follows from \cite[\S~IV.1]{nsw:cohn}.
Statement (b) follows from \cite[\S~II.4, Ex.~7]{nsw:cohn}. Note that both statements also follow from the discussion following Example \ref{ex:quad algebras}.
\end{proof}

\begin{remark}\label{rem:directproduct}
 Let $G_1,G_2$ be two quadratic pro-$p$ groups which occur as maximal pro-$p$ Galois groups
 of fields. 
 By \cite[Thm.~C]{cq:bk}, the direct product $G_1\times G_2$ may occur as a maximal pro-$p$ Galois group
 only if one of the two groups is abelian.
\end{remark}

As we have seen above, quadratic pro-$p$ groups are closed under taking free and direct products. There are other universal ``free product-like'' constructions and, in the rest of this section, we will explore to what extent the class of quadratic pro-$p$ groups is closed under these constructions.

\subsection{Amalgamated products}

Let $G_1$ and $G_2$ be pro-$p$ groups and let $\phi_i\colon H\to G_i$ (for $i\in\{1,2\}$) be continuous
monomorphisms of pro-$p$ groups.
An \emph{amalgamated free pro-$p$ product} (or simply \emph{amalgam}) of $G_1$ and $G_2$ with \emph{amalgamated subgroup} $H$ is defined
to be the pushout
\[
 \xymatrix{ H\ar[r]^{\phi_1}\ar[d]_{\phi_2} & G_1\ar@{-->}[d]^{\psi_1} \\ G_2\ar@{-->}[r]^{\psi_2} & G }
\]
in the category of pro-$p$ groups, which is unique (cf.\ \cite[\S~9.2]{ribzal:book}).
We write $G=G_1\amalg_H G_2$.

An amalgamated free pro-$p$ product $G=G_1\amalg_H G_2$ is said to be \emph{proper}
if the homomorphisms $\psi_i$ ($i = 1, 2$) are monomorphisms.
In that case one identifies $G_1$, $G_2$ and $H$ with their images in $G$.
\begin{remark}
For a proper amalgam $G=G_1\amalg_H G_2$ it follows from  \cite[Prop.~9.2.13(a)]{ribzal:book} that $$\cdd(G)\le \max\{\cdd(G_1),\cdd(G_2),\cdd(H)+1\}.$$
\end{remark}

\begin{example}\label{example:demushkin amalg}
For $p\neq2$, let $G$ be a Demushkin group with presentation
\[
 G=\left\langle x_1,\ldots,x_d\mid x_1^q[x_1,x_2][x_3,x_4]\cdots[x_{d-1},x_d]=1 \right\rangle,
\]
with $d=\dd(G)\geq4$ and $q$ a power of $p$.
Let $G_1,G_2\le G$ be the subgroups generated by 
$x_1,x_2$ and by $x_3,\ldots,x_d$ respectively, and $H\le G$ the pro-cyclic subgroup
generated by $[x_2,x_1]x_1^{-q}=[x_3,x_4]\cdots[x_{d-1},x_d]$.
Then $G$ is (isomorphic to) the proper amalgam $G_1\amalg_H G_2$ (cf.\ \cite[Ex.~9.2.12]{ribzal:book}; the same holds also if $p=2$). Note that $\cdd(G_1)=\cdd(G_2)=1$, but $\cdd(G)=2$.
\end{example}

We come to the proof of the first of the two main results of this section: Theorem~\ref{thm:cohomology amalgam}.

\begin{remark}\label{remark:restriction}
For a pro-$p$ group $G$ and a subgroup $H\le G$, the restriction map 
\[
   \mathrm{res}_{G,H}^\bullet\colon H^\bullet(G,\F_p)\longrightarrow H^\bullet(H,\F_p),
\]
induced by $\mathrm{res}_{G,H}^n$ for every $n\geq0$, is a morphism of graded algebras
(cf.\ \cite[Prop.~1.5.3]{nsw:cohn}).
\end{remark}

\begin{proof}[\textbf{Proof of Theorem~\ref{thm:cohomology amalgam}}]
By hypothesis~(i), the maps $\res_{G_i,H}^1$ are surjective, for $i=1,2$.
Set $V_i=\kernel(\res_{G_i,H}^1)$, and let $W_i\subseteq H^1(G_i,\F_p)$ be a complement for $V_i$ for each $i$.
Then we may identify both $W_1$ and $W_2$ with $H^1(H,\F_p)$ via $\res^1_{G_1,H}$ and $\res_{G_2,H}^1$, respectively, and with an abuse of notation we write $W=W_1=W_2$, so that $H^1(G_i,\F_p)=V_i\oplus W$ for $i=1,2$.
Since $H$ is quadratic, one has an isomorphism of quadratic algebras 
\[
 H^\bullet(H,\F_p)\simeq \frac{\Lambda_\bullet(W)}{(\Omega_H)},
\]
for some subspace $\Omega_H\subseteq \Lambda_2(W)$.
Moreover, by hypothesis~(ii), there exist subspaces $\Omega_i\subseteq\Lambda_2(V_i)\oplus(V_i\wedge W)$ and isomorphisms of quadratic algebras
\[
 H^\bullet(G_i,\F_p)\simeq \frac{\Lambda_2(V_i\oplus W)}{(\Omega_i\cup\Omega_H)}
\]
 for $i=1,2$.
In particular, in $H^\bullet(G_i,\F_p)$ one has a relation $a+b=0$ with $a\in H^1(G_i,\F_p)\wedge V_i$ and $b\in \Lambda_2(W)$ if and only if $a\in \Omega_i$ and $b\in \Omega_H$.

Since $G=G_1\amalg_{H}G_2$ is a proper amalgam, by \cite[Prop.~9.2.13]{ribzal:book} the monomorphisms 
$H\to G_i$ and $G_i\to G$ for $i=1,2$ induce a long exact sequence in cohomology  
\[
 \xymatrix@C=0.5truecm{ 0\ar[r] & H^1(G,\F_p)\ar[r]^-{ f_G^1}& \cdots\ar[r]^-{ f_H^{n-1}} & H^{n-1}(H,\F_p) \ar`r[d]`[l] `[dlll] `[dll] [dll]   \\
&H^n(G,\F_p)\ar[r]^-{ f_G^n}& H^n(G_1,\F_p)\oplus H^n(G_2,\F_p)\ar[r]^-{ f_H^n} & H^n(H,\F_p) \ar`r[d]`[l] `[dlll] `[dll] [dll]   \\
& H^{n+1}(G,\F_p)\ar[r]^-{ f_G^{n+1}} & H^{n+1}(G_1,\F_p)\oplus H^{n+1}(G_2,\F_p)\ar[r]^-{ f_H^{n+1}} &\cdots  &}
\]
 for $n\geq1$, where 
\[  \begin{split}
   f_G^n(\xi) &=\left(\mathrm{res}_{G,G_1}^n(\xi),\mathrm{res}_{G,G_2}^n(\xi)\right), \\
   f_H^n(\eta,\eta')&=\mathrm{res}_{G_1,H}^n(\eta)-\mathrm{res}_{G_2,H}^n(\eta')
 \end{split} \]
for every $\xi\in H^n(G,\F_p)$ and $\eta\in H^n(G_1,\F_p)$, $\eta'\in H^n(G_2,\F_p)$ (cf.\ \cite{gildenribes:amalg}).

Since $G_1$, $G_2$ and $H$ are quadratic, Remark~\ref{remark:restriction} implies
that the restriction maps $\mathrm{res}_{G_1,H}^\bullet$ and $\mathrm{res}_{G_2,H}^\bullet$
are epimorphisms of quadratic algebras. Therefore also the maps $f_H^n$ are surjective for all $n\geq1$.
Thus, 
\[
 H^\bullet(G,\F_p)\cong \F_p\oplus\left(\bigoplus_{n\geq1}\kernel(f_H^n)\right),
\]
where $\F_p$ is the degree 0 part --- note that this is a morphism of graded algebras, as the cup-product commutes with the restriction maps (cf.\ \cite[Prop.~1.5.3]{nsw:cohn}).

For every $n\geq1$, we may decompose the map $f_H^n$ as 
\[
 \xymatrix@C=0.5truecm{ H^n(G_1,\F_p)\oplus H^n(G_2,\F_p)\ar[r]^-{ \varphi_n} & H^n(H,\F_p)\oplus H^n(H,\F_p)\ar[r]^-{\psi_n}
& H^n(H,\F_p)},\]
where $\varphi_n=\mathrm{res}_{G_1,H}^n\oplus\mathrm{res}_{G_2,H}^n$ and $\psi_n=\pi_n-\pi_n'$,
with $\pi_n$ and $\pi_n'$ the canonical projections onto the first and the second summand respectively.
Clearly, $\kernel(\psi_n)$ is the diagonal $\Delta(H^n(H,\F_p))\subseteq H^n(H,\F_p)\oplus H^n(H,\F_p)$.
Moreover, by hypothesis (ii) and by Remark~\ref{remark:restriction},
$\kernel(\mathrm{res}_{G_i,H}^\bullet)$ is generated as ideal of
$H^\bullet(G_i,\F_p)$ by $\kernel(\res_{G_i,H}^1)$ for $i=1,2$, so that
\begin{equation}\label{eq:ker resn Gi}
  \ker(\res_{G_i,H}^n)=\kernel(\res_{G_1,H}^1)\wedge H^{n-1}(G_i,\F_p)
\end{equation}
for each $n\geq1$ and $i=1,2$.
Therefore, one has 
\begin{equation}
  \begin{split}
  \kernel(f_H^n) & = \kernel(\varphi_n)\oplus \Delta(H^n(H,\F_p)) \\
  &\cong (V_1\wedge H^{n-1}(G_1,\F_p))\oplus (V_2\wedge H^{n-1}(G_2,\F_p))\oplus H^n(H,\F_p)
 \end{split}
\end{equation}
for all $n\geq1$ (here we identify $H^n(H,\F_p)$ with $H^n(G_i,\F_p)/\kernel(\res^n_{G_i,H})$).

Now let $A_\bullet$ be the quadratic algebra $\Lambda_\bullet(V_1\oplus W\oplus V_2)/(\Omega_G)$, with $\Omega_G=\Omega_1\oplus\Omega_2\oplus\Omega_H\oplus (V_1\wedge V_2)$.
In particular, one has the isomorphisms of graded algebras $A_\bullet/(V_1)\cong H^\bullet(G_2,\F_p)$,
$A_\bullet/(V_2)\cong H^\bullet(G_1,\F_p)$ and 
\begin{equation}\label{eq:cong W H}
 A_\bullet/(V_1\oplus V_2)\cong \Lambda_\bullet(W)/(\Omega_H)\cong H^\bullet(H,\F_p).
\end{equation}
Moreover, for every $n\geq 1$, we have the isomorphisms of vector spaces
\begin{equation}\label{eq:cong V1V2 HG}
 \frac{V_i\wedge (\Lambda_{n-1}(V_i\oplus W))}{(\Omega_i)_n}\cong V_i\wedge H^{n-1}(G_i,\F_p),
\end{equation}
where $(\Omega_i)_n$ denotes the part of degree $n$ of the ideal $(\Omega_i)$.
Let $$\phi_\bullet\colon A_\bullet\longrightarrow H^\bullet(G_1,\F_p)\oplus H_1(G_2,\F_p)$$
be the morphism of quadratic algebras given by
$ \phi_1\vert_{V_1}=\mathrm{id}_{V_1}\oplus 0$, $ \phi_1\vert_{V_2}=0\oplus \mathrm{id}_{V_2}$, and $\phi_1\vert_{W}=\mathrm{id}_W\oplus\mathrm{id}_W$ (here 0 denotes the 0-map).
Since $\psi_1\circ\phi_1\vert_{V_1\oplus V_2}=0$, and $f_H^1\circ\phi_1\vert_{W}=0$, by quadraticity of $A_\bullet$ one has that the image of $\phi_n$ is contained in $\kernel(f_H^n)$ for every $n\geq1$.
Moreover, by \eqref{eq:ker resn Gi}, \eqref{eq:cong W H} and \eqref{eq:cong V1V2 HG}, the map $\phi_n\colon A_n\to \kernel(f_H^n)$ is an isomorphism.
Therefore, $\phi_\bullet\colon A_\bullet\to \image(\phi_\bullet)$ is an isomorphism of quadratic algebras, and $H^\bullet(G,\F_p)$ is quadratic.
\end{proof}

\begin{remark}\label{rmk:hypotheses pRAAGs}
By duality \eqref{eq:H1}, the restriction maps \eqref{eq:surj cohomology amalgam}, with $k=1$,
are surjective if and only if the monomorphisms $H\to G_1$ and $H\to G_2$ induce monomorphisms 
\[
 \frac{H}{H_{(2)}}\longrightarrow \frac{G_1}{(G_1)_{(2)}},
 \quad \frac{H}{H_{(2)}}\longrightarrow \frac{G_1}{(G_1)_{(2)}}.
\]
In other words, one may find bases $\mathcal{X}_1$ and $\mathcal{X}_2$ for $G_1$ and $G_2$ respectively, such that
$H$ is generated as a subgroup of $G_1$ and $G_2$ by $\mathcal{X}_1\cap H$ and $\mathcal{X}_2\cap H$ respectively.

Let $\mathcal{X}_1$ and $\mathcal{X}_2$ be as above and, for $i\in\{1,2\}$, let $x_h,x_j\in \mathcal{X}_i\cap H$.
By Proposition~\ref{prop:cupproduct}, condition (ii) of the statement of Theorem~\ref{thm:cohomology amalgam}
holds in the following case: if $a_{hj}\neq0$ for some relation $r$ of $G_i$, then there is a relation
$[x_h,x_j]=t$, with $t\in\Phi(H)$. In particular, as it will be clear later, Theorem~\ref{thm:cohomology amalgam} works for generalised $p$-RAAGs: a 
relation of $G_i$ is also a relation of $H$ (see Section~\ref{sec:pRAAGs}).
\end{remark}

The following two examples of proper amalgams of quadratic pro-$p$ groups show that both conditions
of the statement of Theorem~\ref{thm:cohomology amalgam} are necessary.

\begin{example}[\textbf{Condition (i) is necessary}]\label{example:amalg1}
Let $G_1=\langle x_1,y_1\rangle$ and $G_2=\langle x_2,y_2\rangle$ be two 2-generated free
pro-$p$ groups, and set $z_1=[x_1,[x_1,  y_1]]$ and $z_2=[x_2,[x_2,y_2]]$.
Then $G_1$, $G_2$, $\langle z_1\rangle$ and $\langle z_2\rangle$ are all quadratic pro-$p$ groups.
Let $G$ be the proper amalgam $G_1\amalg_{z_1=z_2}G_2$ (the amalgam is proper by \cite[Thm.~3.2]{ribes:amalg}).
Then $G$ is not quadratic by \cite[Cor.~9.2]{CEM}.
Here the restriction maps $\res_{G_i,\langle z_i\rangle}^1$ are trivial for both $i=1,2$:
  indeed $z_i\in\Phi(G_i)$, so that $\alpha(z_i)=0$ for any $\alpha\in H^1(G_i,\F_p)$.
\end{example}

\begin{example}[\textbf{Condition (ii) is necessary}]\label{example:amalg2}
For $p\neq2$ let $G_1,G_2$ be the pro-$p$ groups with minimal presentations
\[
 G_1=\langle x_1,x_2,x_3\mid x_1^p[x_2,x_3]=1\rangle,\quad G_2=\langle x_2,x_3,x_4\mid x_4^p[x_2,x_3]=1\rangle.
\]
Then $G_1,G_2$ are isomorphic free-by-Demushkin pro-$p$ groups (cf.\ \cite{kochzal}).
Also, they are quadratic by \cite[Prop.~4.3]{cq:onerel}.
Let $H\le G_1,G_2$ be the closed subgroup generated by $x_2,x_3$.
Then $H$ is a free 2-generated pro-$p$ group by the Freiheitsatz (cf.\ \cite{romanov:freiheit}) and thus it is quadratic as well.

Set $G=G_1\amalg_H G_2$. By \cite[Ex.~9.2.6(a)]{ribzal:book}, $G$ is a proper amalgam. Let $\{\alpha_1,\ldots,\alpha_4\}$ be a basis of $H^1(G,\F_p)$
dual to $\{x_1,\ldots,x_4\}$. 
Then, since $\cdd(H)=1$, the long exact sequence in cohomology induces an isomorphism in degree 2
\begin{equation}\label{eq:example amalg isoH2}
  \res_{G,G_1}^2\oplus\res_{G,G_2}^2\colon H^2(G,\F_p)\overset{\sim}{\longrightarrow} H^2(G_1,\F_p)\oplus H^2(G_2,\F_p),
\end{equation}
with $H^2(G_1,\F_p)=\langle\res_{G,G_1}^2(\alpha_2\alpha_3)\rangle$ and
$H^2(G_2,\F_p)=\langle\res_{G,G_2}^2(\alpha_2\alpha_3)\rangle$, both isomorphic to $\F_p$.
Moreover, $\res_{G_i,H}^1$ is surjective and
$$\ker(\res_{G_i,H}^2)=H^2(G_i,\F_p)\neq\ker(\res_{G_i,H}^1)\wedge H^1(G_i,\F_p),$$
for both $i=1,2$.
By Proposition~\ref{prop:cupproduct}, the only element of $H^2(G,\F_p)$ generated in degree 1 is $\alpha_2\alpha_3$,
and since $H^2(G,\F_p)$ has dimension 2 by \eqref{eq:example amalg isoH2}, $H^\bullet(G,\F_p)$ is not quadratic.
\end{example}

\subsection{HNN extensions}

Let $G_0$ be a pro-$p$ group and let $\phi\colon A\to B$ be a continuous isomorphism between subgroups
$A,B\le G_0$.
A \emph{pro-$p$ HNN-extension} of $G_0$ with associated subgroups $A$ and $B$ is given by the pro-$p$ group
$G=\mathrm{HNN}(G_0,A,\phi)$, together with an element $t\in G$ and a continuous homomorphism $\psi\colon G_0\to G$ such that
$$t(\psi(a))t^{-1}=\psi\circ\phi(a)$$ for every $a\in A$, which satisfy the following universal property:
for any pro-$p$ group $G'$, any $g\in G'$ and any continuous homomorphism $f\colon G_0\to G'$
satisfying $g(f(a))g^{-1}=f\circ\phi(a)$ for all $a\in A$, there is a unique continuous
homomorphism $\tilde f\colon G\to G'$ with $\tilde f(t)=g$ such that $\tilde f\circ \psi=f$.
Such a pro-$p$ group $G$ is unique (cf.\ \cite[\S~9.4]{ribzal:book}).

If the homomorphism $\psi\colon G_0\to G$ is injective, the HNN extension is said to be \emph{proper}. 

\begin{remark}
Note that, for a proper HNN extension, it follows from  \cite[Prop.~9.4.2(a)]{ribzal:book} that $$\cdd(\mathrm{HNN}(G_0,A,\phi)) \le \max\{\cdd(G_0),\cdd(A)+1\}.$$
\end{remark}

\begin{example}\label{example:demushkin HNN}
 In \cite[Ex.~9.2.12]{ribzal:book}, it is shown that a Demushkin pro-$p$ group is an HNN-extension.
\end{example}

We come to the second main result of this section: Theorem~\ref{thm:HNN}.

\begin{remark}\label{remark:HNN thm}
 Condition (ii) in Theorem~\ref{thm:HNN} amounts to say that the morphism
 $$\bar\phi\colon \frac{A}{A\cap(G_0)_{(2)}}\longrightarrow\frac{B}{B\cap(G_0)_{(2)}},$$
induced by $\phi$, is the restriction of the identity of $G_0/(G_0)_{(2)}$ on $A/A\cap(G_0)_{(2)}$.
\end{remark}

\begin{proof}[\textbf{Proof of Theorem~\ref{thm:HNN}}]
 Since $G=\mathrm{HNN}(G_0,A,\phi)$ is proper, by \cite[Prop.~9.4.2]{ribzal:book}
 the monomorphisms $G_0\to G$ and $A\to G_0$ induce a long exact sequence in cohomology  
\[
 \xymatrix@C=0.8truecm{ 0\ar[r] &\F_p\ar[r] & H^1(G,\F_p)\ar[r]^-{\res_{G,G_0}^1 }& \cdots\ar[r]^-{ f_{G_0}^{n-1}} & H^{n-1}(A,\F_p) \ar`r[d]`[l] `[dlll]_-{\delta^{n-1}} `[dll] [dll]   \\
& & H^n(G,\F_p)\ar[r]^-{\res_{G,G_0}^n}& H^n(G_0,\F_p)\ar[r]^-{ f_{G_0}^n} & H^n(A,\F_p) \ar`r[d]`[l] `[dlll]_-{\delta^n} `[dll] [dll]   \\
& & H^{n+1}(G,\F_p)\ar[r]^-{\res_{G,G_0}^{n+1}} & H^{n+1}(G_0,\F_p)\ar[r]^-{f_{G_0}^{n+1}} &\cdots  &}
\]
 for $n\geq1$, where $f_{G_0}^n=\res_{G_0,A}^n-\phi^*\circ\res_{G_0,B}^n$,
with $\phi^*\colon H^n(B,\F_p)\to H^n(A,\F_p)$  the map induced by $\phi$ --- see also \cite{bieri:HNN}.

Since the map $\phi^*$ commutes with the cup product for every $n\geq1$ and $\alpha_1,\ldots,\alpha_n\in H^1(G_0,\F_p)$,
one has 
\[\begin{split}
   \res_{G_0,A}^n(\alpha_1\cdots\alpha_n) &= \res_{G_0,A}^1(\alpha_1)\cdots\res_{G_0,A}^1(\alpha_n) \\
 &= \phi^*\res_{G_0,B}^1(\alpha_1)\cdots\phi^*\res_{G_0,B}^1(\alpha_n)\\
 &= \phi^*\res_{G_0,B}^n(\alpha_1\cdots\alpha_n),
  \end{split}\]
and hence by hypothesis (ii) the maps $f_{G_0}^n$ are trivial for every $n\geq1$.
Therefore, for every $n\geq2$ one has a short exact sequence of vector spaces
\begin{equation}\label{eq:ses hnn cohomology}
  \xymatrix{ 0\ar[r] & H^{n-1}(A,\F_p)\ar[r]^{\delta^{n-1}} & H^n(G,\F_p)\ar[r]^-{\res_{G,G_0}^n} & H^n(G_0,\F_p)\ar[r] &0 }.
\end{equation}
We will identify $H^n(G_0,\F_p)$ and $H^n(A,\F_p)$ as subspaces of $H^n(G,\F_p)$.

Let $\alpha_t\in H^1(G,\F_p)$ be the generator of $\kernel(\res_{G,G_0}^1)$ --- i.e.\ $\alpha_t$ is dual to $t\in G$.
By Remark~\ref{remark:HNN thm}, for every $a\in A$ one has a relation $[t,a]=a'$, with $a'\in A\cap(G_0)_{(2)}$.
Hence, by Proposition~\ref{prop:cupproduct}, the cup-product $\alpha \alpha_t\in H^2(G,\F_p)$ is not trivial for every
$\alpha\in H^1(G_0,\F_p)$ such that $\res_{G_0,A}^1(\alpha)\neq0$. Therefore
$\image(\delta^{1})=\ker(\res_{G,G_0}^2)=\alpha_t\wedge H^1(A,\F_p)$.
Thus, for every $n\geq1$, one has an isomorphism of vector spaces
\begin{equation}\label{eq:HNN iso vectspa}
  H^n(G,\F_p)=H^n(G_0,\F_p)\oplus\left(\alpha_t\wedge H^{n-1}(A,\F_p)\right)
\end{equation}
and $H^\bullet(G,\F_p)$ is generated in degree 1.
Moreover, from \eqref{eq:HNN iso vectspa} one deduces
\[
  H^n(G,\F_p)\cong \frac{H^n(G_0,\F_p)\oplus (\alpha_t\wedge H^{n-1}(G_0,\F_p))}{\alpha_t\wedge\kernel(\res_{G_0,A}^{n-1})},
\]
for every $n\geq2$.
By hypothesis (i), $\kernel(\res_{G_0,A}^{n-1})=\kernel(\res_{G_0,A}^1)\wedge H^{n-2}(G_0,\F_p)$,
that is, it is generated in degree 1. Thus $\alpha_t\wedge\kernel(\res_{G_0,A}^{n-1})$
is generated in degree 2 for every $n\geq2$.
It follows that $H^\bullet(G,\F_p)$ is a quadratic algebra.
\end{proof}

The following is an example of a proper HNN-extension satisfying all the hypothesis of Theorem~\ref{thm:HNN}
and it is a new example of quadratic pro-$p$ group obtained in this way.

\begin{example}\label{example:HNN1}
 Let $G_0 =\langle x,y,z \rangle$ be free abelian pro-$p$ group. 
Set $A= \langle x,y \rangle \le G_0$ and let $\phi\colon A\to A$ be the isomorphism induced by $x\mapsto xy^p$
and $y\mapsto yz^p$. Consider $G=\mathrm{HNN}(G_0,A,\phi)$. It follows easily using \cite[Prop.~9.4.3(2)]{ribzal:book} that $G$ is a proper HNN-extension.
Thus, $G$ has a presentation
\[
 G=\langle x,y,z,t\mid [x,y]=[y,z]=[z,x]=1,[x,t]=y^p,[y,t]=z^p\rangle.
\]
Let $\{\alpha_x,\alpha_y,\alpha_z\}\subseteq H^1(G_0,\F_p)$ be a basis dual to $\{x,y,z\}$.
Then $\kernel(\res_{G_0,A}^1)=\langle\alpha_z\rangle$ and 
\[\kernel(\res_{G_0,A}^2)=\langle\alpha_x\alpha_z,\alpha_y\alpha_z\rangle=
\kernel(\res_{G_0,A}^1)\wedge H^1(G_0,\F_p).\]
Moreover, $\phi^*\colon H^1(A,\F_p)\to H^1(A,\F_p)$ is the identity.
Therefore, $G$ is quadratic pro-$p$ group.
\end{example}

On the other hand, Theorem~\ref{thm:HNN} can also be used to show that certain HNN-extensions are not proper.

\begin{example}
 Let $G$ be the pro-$p$ group with presentation $$\langle x_1,x_2,x_3,x_4 \mid [x_1,x_2][x_3,x_4]=1,[x_1,x_3]=1,[x_1,x_4][x_2,x_3]=1,[x_2,x_4]=1 \rangle.$$ 
Let $\{\alpha_1,\alpha_2,\alpha_3,\alpha_4\}$ be a basis of $H^1(G,\F_p)$ dual to $\{x_1,x_2,x_3,x_4\}$.
Then we have $\alpha_1\alpha_2=\alpha_3\alpha_4$, $\alpha_1\alpha_4=\alpha_2\alpha_3$ and therefore $\{\alpha_1\alpha_2,\alpha_1\alpha_3,\alpha_1\alpha_4,\alpha_2\alpha_4\}$
is a basis of $H^2(G,\F_p)$ by Proposition~2.3.
It is clear that $\alpha_i\alpha_j\alpha_k=0$ for any triple $(i,j,k)$; also it is easy to check that conditions (i)-(iii) of Theorem~\ref{thm:HNN} are satisfied.

Let $N,H\le G$ be the subgroups generated by $x_2,x_3,x_4$ and by $x_1$, respectively.
Then $N$ is normal in $G$. The short exact sequence \eqref{eq:ses cohom GGS} implies that $\cdd(G)\geq3$.
Thus $G$ is not quadratic.

Notice that we can realise $G$ as $\mathrm{HNN}(G_0,G_0,\phi)$ with $$G_0 = \langle x_2,x_4\mid [x_2,x_4]=1\rangle \ast \langle x_3 \rangle \cong (\Z_p\times \Z_p) \ast\Z_p ,\quad t=x_1$$ and $\phi(x_2)=x_2[x_3,x_4]$, $\phi(x_3)=x_3$, $\phi(x_4)=x_4[x_2,x_3]$.
Hence, $\mathrm{HNN}(G_0,G_0,\phi)$ is not proper by Theorem~\ref{thm:HNN}.
\end{example}

The following is a list of examples of proper HNN-extensions which are not quadratic
pro-$p$ groups, each of which does not satisfy one of the hypothesis of Theorem~\ref{thm:HNN}. That the given HNN-extensions are proper one can show easily using  \cite[Prop.~9.4.3(2)]{ribzal:book}.

\begin{example}[\textbf{Condition (i) is necessary}]\label{ex:HNNcondi}
Let $G_0 = \langle x,y \rangle$ be a free abelian pro-$p$ group. 
Set $A=\langle x^p \rangle$ and $B=\langle y^p \rangle$ and let $\phi\colon A\to B$ be the isomorphism induced by $x^p\mapsto y^p$. Consider $G=\mathrm{HNN}(G_0,A,\phi)$.
Thus $G$ has a presentation
\[
 G=\langle x,y,t\mid [x,y]=1,[x^p,t]=(x^{-1}y)^p\rangle,
\]
and $G$ is not quadratic, as Proposition~\ref{prop:relations} is not satisfied.
Indeed, the map $\res_{G_0,A}^1$ is not surjective.
\end{example}

\begin{example}[\textbf{Condition (ii) is necessary}]\label{ex:HNNcondii}
 Let $G_0$ be the pro-$p$ group with minimal presentation
\[
 G_0=\langle x,y,z\mid x^p[y,z]=1\rangle,
\]
let $A=B\le G_0$ be the subgroup generated by $y,z$ (in particular, $A$ is a free 2-generated pro-$p$
group by the Freiheitsatz, cf.\ \cite{romanov:freiheit}) and let $\phi\colon A\to A$ be the identity.
Consider $G=\mathrm{HNN}(G_0,A,\phi)$.
Thus, $G$ has a presentation
\[
 G=\langle x,y,z,t\mid x^p[y,z]=[t,y]=[t,z]=1 \rangle.
\]
Let $\{\alpha_x,\ldots,\alpha_t\}\subseteq H^1(G,\F_p)$ and, with a slight abuse of notation, consider $\alpha_x,\alpha_y,\alpha_z$
as elements of $H^1(G_0,\F_p)$.
Then $\kernel(\res_{G_0,A}^1)=\langle\alpha_x\rangle$ and 
\[
  \kernel(\res_{G_0,A}^2)=\langle\alpha_y\alpha_z\rangle\neq\kernel(\res_{G_0,A}^1)\wedge H^1(G_0,\F_p).
  \] 
On the other hand, the long exact sequence in cohomology induced by the HNN-extension
implies that $H^2(G,\F_p)=\langle\alpha_y\alpha_z,\alpha_t\alpha_y,\alpha_t\alpha_z\rangle$,
whereas $H^3(G,\F_p)=0$ as $\cdd(G_0)=2$ and $\cdd(A)=1$, so that $G$ is not quadratic.
\end{example}

\begin{example}[\textbf{Condition (iii) is necessary}]\label{ex:HNNcondiii}
Let $G_0 =\langle x,y \rangle$ be a free abelian pro-$p$ group. 
Set $A=\langle x \rangle$, $B=\langle y \rangle$ and let $\phi\colon A\to B$ be the isomorphism induced by $x\mapsto y$. Consider $G=\mathrm{HNN}(G_0,A,\phi)$.
Thus, $G$ has a presentation
\[
 G=\langle x,y,t\mid [x,y]=1,[x,t]=x^{-1}y\rangle=\langle x,t\mid [x,[x,t]]=1\rangle
\]
and $G$ is not quadratic, as Proposition~\ref{prop:relations} is not satisfied.
Indeed, the map $f_{G_0}^1$ is not trivial.
\end{example}

\section{Analytic pro-\texorpdfstring{$p$}{p} groups}\label{sec:quadraticpadic}

A pro-$p$ group $G$ is said to be \emph{powerful} if $p \geq 3$ and
$[G,G]\leq G^p$, or $p=2$ and $[G,G]\leq G^4$.  Here, $[G,G]$ and
$G^p$ denote the commutator subgroup and the 
subgroup generated by all $p$th powers. Recall that the \emph{descending $p$-central series of $G$} is defined inductively by $P_1(G)=G$ and $P_{n+1}(G) = P_n(G)^p [P_n(G),G]$ for $n\ge 1$.
A pro-$p$ group $G$ is called \emph{uniform}  if
it is finitely generated, powerful and  $$\lvert P_i(G) : P_{i+1}(G) \rvert = \lvert G : P_2(G) \rvert$$
for all $i \in \N$.
It is worth noting that a powerful finitely generated pro-$p$ group is uniform if and only if it is torsion-free (see \cite[Thm.~4.5]{ddms:padic}).
Uniform pro-$p$ groups play an important role in the theory of $p$-adic Lie groups (see \cite{ddms:padic}). 

In light of Lazard's work we have the following (cf.\ \cite{lazard:padic}, see also \cite[Thm.~5.1.5]{sw:cpg}).

\begin{proposition}\label{prop:uniform}
 Let $G$ be a uniform pro-$p$ group.
 Then $$H^\bullet(G,\F_p)\cong\Lambda_\bullet H^1(G,\F_p).$$
\end{proposition}

We can now prove Theorem~\ref{thm:quadratic analytic}. 

\begin{proof}[\textbf{Proof of Theorem~\ref{thm:quadratic analytic}}]
If $G$ is uniform, then the claim follows by Proposition~\ref{prop:uniform}.

Assume now that $G$ has quadratic $\F_p$-cohomology.
Since $G$ is torsion-free, one has $\dd(G)\leq\dim(G)$, by \cite[Prop.~1.2]{klopsch:rank}. Moreover, $G$ is a Poincar\'e-Duality pro-$p$ group
so that $\dim(G)=\cdd(G)$. Now by (\ref{eq:inequalities}), 
\[
 \dd(G)\leq\dim(G)=\cdd(G)\leq \dd(G),
\]
 so $\cdd(G)=\dd(G)$ and $H^\bullet(G,\F_p)\cong\Lambda_\bullet(H^1(G,\F_p))$
(cf.\ \cite[Prop.~4.3]{cq:bk}).
In particular, $G$ is powerful by \cite[Thm.~5.1.6]{sw:cpg}, and hence uniform.
\end{proof}

The following examples show that the above theorem cannot be extended to $p=2$, even for torsion-free groups.

\begin{example}\label{ex:analyticp21}
Let $G_d= N_d\rtimes H$ be the pro-$2$ group with $$N_d = \langle x_1,\ldots,x_d \mid [x_i,x_j]=1,\ 1\le i<j\le d \rangle \cong \mathbb{Z}_2^d,\quad H=\langle y \mid \ \rangle \cong \mathbb{Z}_2$$ and $y^{-1}x_i y=x_i^{-1}$, for $1\le i\le d$. 
 Then $G_d$ is $2$-adic analytic and torsion-free, but it is not uniform.
Moreover the group $G_d$ is quadratic, by \cite[Thm.~3.16]{qw:cyclotomic}.
Thus, the hypothesis $p\ge 3$ in the above theorem cannot be removed.
 \end{example}

There are also infinite quadratic pro-$2$ groups with torsion. 
 
 \begin{example}\label{ex:analyticp22}
The infinite dihedral pro-$2$ group $D_{\infty}= \Z_2\rtimes C_2$, with action given by inversion, is isomorphic to the free pro-$2$ product $C_2\ast C_2$. By Proposition~\ref{prop:cohomology freeproduct}, we have 
\[  H^\bullet(G,\F_2) \cong H^\bullet(C_2,\F_2)\oplus H^\bullet(C_2,\F_p)   \cong \F_2[X_1]\oplus \F_2[X_2]\]
with $X_1,X_2$ indeterminates. Hence $G$ is $2$-adic analytic and quadratic, but it contains torsion elements.
\end{example}

\begin{remark}
Since there are uncountably many uniform pro-$p$ groups which are not {com\-men\-su\-ra\-ble} (cf.\ \cite{snopche:uncountably}), we have that there are uncountably many non-commensurable quadratic pro-$p$ groups. Taking free products, it is then easy to see that there are also uncountably many \emph{non-isomorphic} quadratic non-uniform pro-$p$ groups. 
\end{remark}

\section{generalised \texorpdfstring{$p$}{p}-RAAGs}\label{sec:pRAAGs}

Right angled Artin groups (RAAGs for short) are a combinatorial construction that plays a prominent role in Geometric Group Theory. These can be defined as (abstract) groups given by a presentation where all relators are commutators of generators. One possible way to obtain a pro-$p$ group out of a RAAG is to consider the pro-$p$ completion of the group; this yields pro-$p$ groups whose structure remains quite close to that of a RAAG. In this section we introduce a generalised construction of RAAGs for pro-$p$ groups. For more evidence of the novelty and flexibility of this construction see Section~\ref{sec:triang}.

\subsection{\texorpdfstring{$p$}{p}-Graphs and \texorpdfstring{$p$}{p}-RAAGs}\label{sec:pgraphpraag}
 
 We will state some conventions that we will keep for the rest of the article.

An \emph{(oriented) graph} is a couple $\mathcal{G}=(\sV,\sE)$ where $\sV$ is a finite set, whose elements
are called \emph{vertices}, and $\sE \subseteq\sV^{2}$, whose elements
are called \emph{edges}.
For an edge $e=(x_1,x_2)\in \sE$, $o(e):=x_1$ and $t(e):=x_2$ are called the \emph{origin} and the \emph{terminus} 
of $e$, respectively. The \emph{opposite edge} of $e=(x_1,x_2)\in \sE$ is $\overline{e}:=(x_2,x_1) \in\sV^2$.
We denote the set of \emph{opposite edges} of edges in $\mathcal{G}=(\sV,\sE)$ as $\overline{\sE}$. 

Let $\mathcal{G}=(\sV,\sE)$ be a graph. A \emph{loop} in $\mathcal{G}$ is an edge $e\in \sE$ with $\vert\{o(e),t(e)\}\vert=1$. An \emph{(unoriented) circuit of lenght $2$} in $\mathcal{G}$ is couple $\{e,\overline{e}\}$ for $e\in \mathcal{G}$. An \emph{(unoriented) circuit} in $\mathcal{G}$ is a sequence of distinct edges $e_1,\ldots,e_n$
with $n\ge 3$ such that $\{o(e_i), t(e_{i})\} \cap \{o(e_{i+1}), t(e_{i+1})\} \neq \emptyset$, for $i=1,\ldots,n-1$
and $\{o(e_n), t(e_{n})\} \cap \{o(e_{1}), t(e_{1})\} \neq \emptyset$. 
A graph $\mathcal{G}$ is said to be \emph{combinatorial} if it has no loops and no circuits of lenght $2$. 
Note that, in particular, a combinatorial graph has a natural ``orientation'', i.e.\ only one of the pairs $(x_1,x_2)$, $(x_2,x_1)$ with $x_1,x_2 \in \sV$ can appear in $\sE$ and $\sE \cap \overline{\sE}=\emptyset$.
A \emph{$p$-labelling} of a combinatorial graph $\mathcal{G}=(\sV,\sE)$ is a function $f:\sE \to p\mathbb{Z}_p \times p\mathbb{Z}_p$ if $p\ge 3$, and $f:\sE \to 4\mathbb{Z}_2 \times 4\mathbb{Z}_2$ if $p=2$. 

\begin{definition}\label{def:graph}
Let $\mathcal{G}=(\sV,\sE)$ be a combinatorial graph.
\begin{itemize}
 \item[(a)] The graph $\mathcal{G}$ is said to be \emph{complete} if $\overline{\sE}=\sV^2 \smallsetminus (\sE  \cup \{(v,v)\mid v\in \sV\})$.
 \item[(b)] A couple $\mathcal{G}' = (\sV',\sE')$ with $\sV'\subseteq \sV$ and $\sE' \subseteq \sE$ is said to be a \emph{subgraph} of $\mathcal{G}$. 
 \item[(c)] A subgraph $\mathcal{G}'= (\sV',\sE')$ of $\mathcal{G}$ is said to be \emph{full} if $\sE'=\sE\cap(\sV')^2$.
\item[(d)] A full subgraph $\calG'$ of $\calG$ is said to be a \emph{clique} of $\calG$ if $\calG'$ is complete.
\end{itemize}
\end{definition}

\begin{definition}
A \emph{$p$-graph} is a couple $\Gamma=(\mathcal{G},f)$ where $\mathcal{G}=(\sV,\sE)$ is a combinatorial graph and $f$ is a $p$-labelling of $\mathcal{G}$.
\end{definition}

Throughout the paper all graphs and $p$-graphs will be finite. 

\begin{definition}\label{def:pRAAGs}
Let $\Gamma= (\sV,\sE,f)$ be a $p$-graph with $p$-labelling $(f_1,f_2):\sE\to p\mathbb{Z}_p \times p\mathbb{Z}_p$. 
 The \textit{generalised Right Angled Artin {pro-$p$} group} ($p$-RAAG for short) associated to $\Gamma$, denoted by $G_\Gamma$, is the pro-$p$ group defined by the following pro-$p$ presentation: 
\begin{equation}\label{eq:pRAAGdefn}
 G_\Gamma = \pres{\sV}{[o(e),t(e)]= o(e)^{f_1(e)} t(e)^{f_2(e)}\ \text{for } e\in \sE} 
\end{equation}
\end{definition}

We present a couple of examples to clarify the definition.

\begin{example}
Let $\mathcal{G}$ be a graph, let $c\equiv (0,0)\in p\mathbb{Z}_p \times p\mathbb{Z}_p$ be the constant $p$-labelling on $\mathcal{G}$ and set $\Gamma=(\mathcal{G},c)$; then $G_\Gamma$ is the pro-$p$ completion of the abstract RAAG associated to $\calG$. 
\end{example}
\begin{example}\label{ex:pRAAGs}
Let $\Gamma_1$ and $\Gamma_2$ be the $p$-graphs

\begin{center}
\begin{minipage}{0.1\textwidth}
$\Gamma_1 =$
\end{minipage}
\begin{minipage}{0.15\textwidth}
\[
   \xymatrix{x_1 \ar[r]^{(a,b)} & x_2}
\]
\end{minipage}
\begin{minipage}{0.2\textwidth}
 \begin{center} and \end{center}
\end{minipage}
\begin{minipage}{0.1\textwidth}
\qquad $\Gamma_2 =$
\end{minipage}
\begin{minipage}{0.2\textwidth}
 \[
    \xymatrix{ & x_2 & \\ x_1 \ar[ur]^{(\lambda,p^2)} \ar[rr]_{(0,0)}&  & x_3  }
 \]
\end{minipage}
\end{center}

\noindent with $a,b,\lambda\in p\mathbb{Z}_p$. Then the pro-$p$ groups $G_{\Gamma_1}$ and $G_{\Gamma_2}$ are defined by the presentations
\[
G_{\Gamma_1} = \pres{x_1,x_2}{[x_1,x_2] = x_1^a x_2^b} \text{\quad and}
\]
\[
G_{\Gamma_2} =  \pres{x_1,x_2,x_3}{[x_1,x_2]= x_1^{\lambda} x_2^{p^2},\ [x_1,x_3] = 1},
\]
respectively.
\end{example}

We will see in Lemma~\ref{lem:isom}, that $G_{\Gamma_1}$ is a $2$-generated Demushkin group, hence a uniform pro-$p$ group.

\begin{remark}\label{rem:pRAAGs H2}
 One of our main motivations to introduce generalised $p$-RAAGs is that $G_\Gamma$ satisfies the equivalent conditions of Proposition~\ref{prop:relations} and Remark~\ref{rmk:hypotheses pRAAGs}. 
\end{remark}

\subsection{Cohomology of quadratic \texorpdfstring{$p$}{p}-RAAGs}

For a set $S$, we will denote by $\F_p\langle S \rangle$ the vector space with basis $S$.

Let $\Gamma=(\mathcal{G},f)$ be a $p$-graph. In this subsection we will show that the $\mathbb{F}_p$-cohomology of a quadratic $p$-RAAG $G_\Gamma$ is completely determined by its ``underlying graph'' $\mathcal{G}$. Recall that, the opposite graph $\mathcal{G}^{\mathrm{op}} = (\sV^{\mathrm{op}},\sE^{\mathrm{op}})$ is defined by $\mathcal{V}^{\mathrm{op}}=\mathcal{V}$, and $\mathcal{E}^{\mathrm{op}}=\mathcal{V}^2\smallsetminus (\mathcal{E}\cup \overline{\sE})$. For instance,
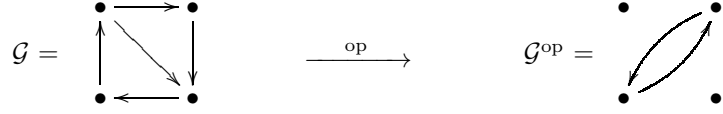
\[
 \begin{minipage}{0.1\textwidth}
    \begin{center}$\mathcal{G}$ = \end{center}
   \end{minipage}
   \begin{minipage}{0.3\textwidth}
\xymatrix{ \bullet \ar[r] \ar[dr] & \bullet \ar[d] \\ \bullet \ar[u] & \bullet \ar[l]}   
   \end{minipage}
   \begin{minipage}{0.3\textwidth}
    \begin{center}$\xrightarrow{\phantom{aaa}\mathrm{op}\phantom{aaa}}$ \end{center}
   \end{minipage}
   \begin{minipage}{0.1\textwidth}
    \begin{center}$\mathcal{G}^{\mathrm{op}}$ = \end{center}
   \end{minipage}
  \begin{minipage}{0.3\textwidth}
\xymatrix{ \bullet & \bullet \ar@/_/[dl] \\ \bullet \ar@/_/[ur]  & \bullet }   
\end{minipage}
 \]

\begin{definition}\label{def:gammaop}
 
Let $\mathcal{G}=(\sV,\sE)$ be a graph. Define the algebra $\Lambda_\bullet(\mathcal{G}^{\mathrm{op}}) = Q(\F_p\langle\sV\rangle,\Omega)$ with
\[
 \Omega=\{v\otimes w + w\otimes v \mid (v,w)\notin\sE\cup \overline{\sE}\}\subseteq \F_p\langle\sV\rangle^{\otimes 2}
\]
(cf.\ \cite[\S~4.2.2]{weigel:koszul}).
Namely, we kill the wedge product of two vertices if they are not connected in $\mathcal{G}$.
In particular, $\Lambda_\bullet(\mathcal{G}^{\mathrm{op}})$ is quadratic and graded-commutative.
 \end{definition}

Let $G_\Gamma$ be a generalised $p$-RAAG with associated $p$-graph $\Gamma=(\mathcal{G},f)$ and underlying graph $\mathcal{G}=(\sV,\sE)$.
Clearly, by \eqref{eq:H1} one has $H^1(G_\Gamma,\F_p)\cong\Lambda_1(\mathcal{G}^{\mathrm{op}})$.
In degree $2$ one has the following.

\begin{lemma}\label{lemma:H2 pRAAGs}
 Let $G_\Gamma$ be a generalised $p$-RAAG with associated $p$-graph $\Gamma=(\mathcal{G},f)$ and underlying graph $\mathcal{G}=(\sV,\sE)$.
 Then $H^2(G_\Gamma,\F_p)\cong \Lambda_2(\mathcal{G}^{\mathrm{op}})$.
\end{lemma}

\begin{proof}
 Set $d=\dd(G_\Gamma)$. Let $\mathcal{X}=\{x_1,\ldots,x_d\}$ and $\mathcal{R}$ be the basis and
  the set of defining relations induced by $\Gamma$, respectively. Also let $\mathcal{X}^*=\{\alpha_1,\ldots,\alpha_d\}$ be the dual basis of $\mathcal{X}$ in $H^1(G_\Gamma,\F_p)$.
Then by Proposition~\ref{prop:cupproduct} one has 
\begin{equation}\label{eq:H2alpha}
 H^2(G_\Gamma,\F_p)=\bigoplus_{\substack{i<j\\(x_i,x_j)\in\sE}}\F_p\cdot\alpha_i\alpha_j,
\end{equation}
with $\alpha_j\alpha_i=\alpha_i\alpha_j$ for all $i,j$.
Finally, \eqref{eq:H2alpha} coincides with $\Lambda_2(\F_p\langle\sV\rangle)/(\Omega)$,
with $\Omega$ as in Definition~\ref{def:gammaop}.
\end{proof}

Once we know that a $p$-RAAG is quadratic, the previous lemma completely determines the $\mathbb{F}_p$-cohomology. 

\begin{proof}[\textbf{Proof of Theorem~\ref{thm:pRAAGs quadratic}}]
 The proof follows at once from the definition of quadraticity and Lemma~\ref{lemma:H2 pRAAGs}.
\end{proof}

In the setting of Theorem~\ref{thm:pRAAGs quadratic}, we may describe the cohomology algebra of $G_\Gamma$ as follows.
Let $\mathcal{X}$ be a basis of $G_\Gamma$ induced by the homomorphism $F\to G_\Gamma$ in \eqref{eq:presentation} and let $\mathcal{X}^*=\{\alpha_1,\ldots,\alpha_d\}$ be its dual basis of $H^1(G_\Gamma,\F_p)$. Fix an integer $n$ with $1\leq n\leq \cdd(G_\Gamma)$.
Then, by Theorem~\ref{thm:pRAAGs quadratic}, we have
$$\beta=\alpha_{i_1}\cdots\alpha_{i_n}\neq0,\qquad\text{with }1\leq i_1<\ldots<i_n\leq n,$$
if, and only if, $(x_{i_j},x_{i_l})\in\sE$ for every
$i_j,i_l\in I=\{i_1,\ldots,i_n\}$.
Namely, $\beta$ is not trivial in $H^n(G_\Gamma,\F_p)$ if, and only if, there exists a clique
$\mathcal{G}_I$ of $\mathcal{G}$ with $\sV(\mathcal{G}_I)=\{x_{i_1},\ldots,x_{i_n}\}$.
Let us denote by $\mathrm{Cl}_n(\Gamma)$ the set of $n$-cliques in $\Gamma$. In particular, one has the following.

\begin{coro}
Let $G_\Gamma$ be a quadratic $p$-RAAG with associated $p$-graph $\Gamma=(\mathcal{G},f)$ and underlying graph $\mathcal{G}=(\sV,\sE)$. For $1\leq n\leq\cdd(G_\Gamma)$ and $1\leq i_1<\ldots<i_n\leq n$, the assignment
\[ 
 \alpha_{i_1}\cdots\alpha_{i_n}\longmapsto \begin{cases}\mathcal{G}_I & \text{if }\alpha_{i_1}\cdots\alpha_{i_n}\neq0,\\
                                            0  & \text{if }\alpha_{i_1}\cdots\alpha_{i_n}=0, \end{cases} \]
with $\mathcal{G}_I$ the $n$-clique of $\Gamma$ defined as above,
induces an isomorphism of vector spaces $$H^n(G_\Gamma,\F_p)\overset{\sim}{\longrightarrow}\F_p\langle\mathrm{Cl}_n(\Gamma)\rangle.$$
\end{coro}

Since we have a clear picture of the cohomology ring of a quadratic $p$-RAAG,
we can describe more precisely how the underlying graph influences the cohomology.

\subsection{Triangle-free \texorpdfstring{$p$}{p}-RAAGs}

In this section we prove that several graphs always yield quadratic pro-$p$ groups. 
The graph $\mathcal{G}=(\sV,\sE)$ is said to be \emph{triangle-free}, if there are no triples $\{x_i,x_j,x_h\}\subseteq\sV$
such that $$(x_i,x_j),(x_i,x_h),(x_j,x_h)\in\sE \cup\overline{\sE}.$$

\begin{proof}[\textbf{Proof of Theorem~\ref{thm:pRAAGs mild}}]
Assume $\mathcal{G}$ is triangle-free.
By Remark~\ref{rem:pRAAGs H2}, $G_\Gamma$ satisfies the conditions of Proposition~\ref{prop:relations}.
Consider the quadratic algebra $A_\bullet(\mathcal{G}) =Q(\F_p\langle \sV \rangle,\Omega)$ where $$ \Omega = \{x_i\otimes x_j -x_j\otimes x_i \mid i<j,\ (x_i,x_j) \in \sE\cup \overline{\sE}\}.$$
By the definition of a $p$-RAAG, the coefficients $a_{ij}$ from \eqref{eq:shape r} are either zero or one, depending on whether $(x_i,x_j) \in \sE$.
Hence $A_\bullet(\mathcal{G}) = M_\bullet(G_\Gamma)$ from Definition~\ref{def:mild}.
   By \cite[\S~4.2.2]{weigel:koszul}, one has an equality of formal power series
\begin{equation*}
 \sum_{n\geq0}\dim(A_n(\mathcal{G}))\cdot T^n=\frac{1}{1-dT+rT^2},
\end{equation*}
where $d=|\sV|$ and $r=|\sE|$. This yields (ii).

Condition (ii) implies (iii) by \cite[Thm.~2.12]{gartner:mild}.

Assume that $\cdd(G_\Gamma)=2$, and suppose that $\Gamma$ contains a triangle $T=(\mathcal{T},f\vert_{\mathcal{T}})$
as a full subgraph, with $\sV(\mathcal{T})=\{x_1,x_2,x_3\}$.
Let $H$ be the subgroup of $G_\Gamma$ generated by $x_1,x_2,x_3$.
Then $\cdd(H)\leq\cdd(G_\Gamma)$ by \cite[Prop.~3.3.5]{nsw:cohn} and $H$ is powerful.
If $H$ is torsion-free, then $\cdd(H)=3$ by Proposition~\ref{prop:uniform} --- in contradiction with (iii).
On the other hand, if $G$ has non-trivial torsion, then $\cdd(H)=\infty$ --- again contradicting (iii). Thus, (iii) implies (i).

Finally, if the three conditions hold, Proposition~\ref{prop:mild cd2} and Lemma~\ref{lemma:H2 pRAAGs}
imply that $H^\bullet(G_\Gamma,\F_p)$ is a quadratic algebra.
\end{proof}

For example, we can show that every ``cycle'' $p$-graph yields a quadratic pro-$p$ group.

\begin{example}
 Let $G$ be the pro-$p$ group with minimal presentation
 \[
  G=\left\langle x_1,\ldots,x_d\mid x_1^p[x_1,x_2]=x_2^p[x_2,x_3]=\ldots=x_d^p[x_d,x_1]=1\right\rangle,
 \]
with $d\geq4$.
By Theorem~\ref{thm:pRAAGs mild}, $G$ is quadratic.
Note that the abelianization $G/[G,G]$ is a $p$-elementary abelian group; moreover, $G$ does not occur 
as a maximal pro-$p$ Galois group by Theorem~\ref{thm:galois pRAAGs}.
\end{example}

The above theorem shows that many $p$-RAAGs are quadratic, since every $p$-graph with triangle-free underlying graph yields a quadratic group. The precise magnitude of triangle-free graphs was calulated by Erd\H{o}s, Kleitman and Rothschild; we record their result below.

\begin{theorem}{\cite{MR0463020}}\label{lem:mostpRAAGs}
 The number of triangle-free graphs on $n$ vertices is asymptotic to $2^{n^2/4 + o(n^2)}$ for $n$ tending to infinity. 
\end{theorem}

Proposition~\ref{prop:mild cd2} raises the following questions, one the ``dual'' of the other.
 
\begin{ques}
 Is every mild pro-$p$ group satisfying Proposition~\ref{prop:relations} quadratic?
 \end{ques}
 \begin{ques}
 Is every finitely presented quadratic pro-$p$ group of cohomological dimension 2 mild?
\end{ques}

By Theorem~\ref{thm:pRAAGs mild}, the above questions have a positive answer for $p$-RAAGs.

\subsection{Triangle-ful \texorpdfstring{$p$}{p}-RAAGs}\label{sec:trianglefull}

In the previous section we saw that a triangle-free $p$-graph always yields a quadratic pro-$p$ group. In particular, the $p$-RAAG associated to a triangle-free $p$-graph is always torsion-free. This is also the case if every edge of $\Gamma$ is labelled by $(0,0)$, i.e.\ $G_\Gamma$ is the pro-$p$ completion of a RAAG.

On the other hand, it turns out that general $p$-RAAGs have a surprisingly rich structure. For instance, it is possible for a $p$-RAAG to be finite.

\begin{example}\label{ex:mennicke}
The $p$-RAAG $G_\Gamma$ associated to the $p$-graph
 
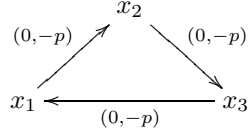
\[
   \xymatrix{& x_2 \ar[dr]^{(0,-p)} & \\ x_1 \ar[ur]^{(0,-p)} &  & x_3 \ar[ll]^{(0,-p)}}
 \]
is given by the presentation $$G=\left\langle x_1, x_2, x_3 \mid [x_1,x_2]=x_2^{-p},  [x_2,x_3]=x_3^{-p},  [x_3,x_1]=x_1^{-p} \right\rangle.$$
This is a finite $p$-group (see \cite[\S~4.4, Ex.~2(e)]{ser:gal}).
\end{example}

Of course a $p$-RAAG as in Example~\ref{ex:mennicke} cannot be quadratic, by Proposition~\ref{prop:properties quadratic} and Remark~\ref{rmk:propCp2}. Hence we need to somehow exclude the possibility that ``triangles collapse the group''. One such condition is given by the following definition.

\begin{definition}\label{defn:nondeg}
The $p$-RAAG $G_\Gamma$ associated to the $p$-graph $\Gamma=(\sV,\sE,f)$ is said to be \emph{non-degenerate} if there exist a subset $\widetilde{\sE} \subseteq \sV^2$ and a $p$-labelling $\widetilde{f}: \widetilde{\sE} \to p\mathbb{Z}_p \times p\mathbb{Z}_p$ such that
 \begin{enumerate}
  \item  $\widetilde{\mathcal{G}} = (\sV,\widetilde{\sE})$ is a combinatorial graph,
  \item  $\sE \subseteq \widetilde{\sE}$ and $\widetilde{f}_{\vert\sE} = f$,
  \item  $G_{(\widetilde{\mathcal{G}},\widetilde{f})}$ is a uniform pro-$p$ group.
 \end{enumerate}
\end{definition} 

The above definition just says that a $p$-RAAG is non-degenerate if its $p$-graph can be ``completed'' to the combinatorial $p$-graph of a uniform pro-$p$ group.

\begin{example}
We will see later that the $p$-RAAG $G_{\Gamma_1}$ from Example~\ref{ex:pRAAGs} is a uniform pro-$p$ group. Hence it is non-degenerate.
 
 We can complete the $p$-graph $\Gamma_2$ from Example~\ref{ex:pRAAGs} as below; see Section~\ref{sec:triang} for the details. 
 \[
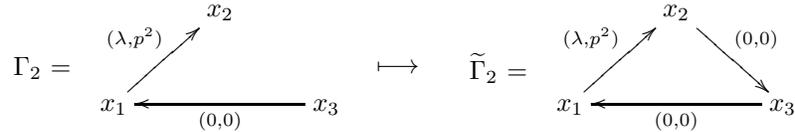
 \begin{minipage}{0.1\textwidth}
    \begin{center}$\Gamma_2$ = \end{center}
   \end{minipage}
   \begin{minipage}{0.3\textwidth}
    \xymatrix{& x_2 & \\ x_1 \ar[ur]^{(\lambda,p^2)} &  & x_3 \ar[ll]^{(0,0)}}
   \end{minipage}
   \begin{minipage}{0.1\textwidth}
    \begin{center}$\longmapsto$ \end{center}
   \end{minipage}
   \begin{minipage}{0.1\textwidth}
    \begin{center}$\widetilde{\Gamma}_2$ = \end{center}
   \end{minipage}
  \begin{minipage}{0.3\textwidth}
    \xymatrix{& x_2 \ar[dr]^{(0,0)} & \\ x_1 \ar[ur]^{(\lambda,p^2)} &  & x_3 \ar[ll]^{(0,0)}}
   \end{minipage}
 \]
 On the other hand, again using the methods from Section~\ref{sec:triang}, one can show that the $p$-graph $\Gamma_3$ below  cannot be completed to give a uniform group.
 \[
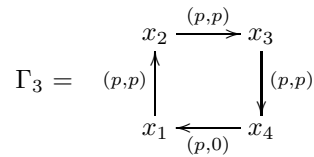
 \begin{minipage}{0.1\textwidth}
    \begin{center}$\Gamma_3$ = \end{center}
   \end{minipage}
   \begin{minipage}{0.3\textwidth}
    \xymatrix{ x_2 \ar[r]^{(p,p)} & x_3 \ar[d]^{(p,p)} \\ x_1 \ar[u]^{(p,p)} & x_4 \ar[l]^{(p,0)}}
   \end{minipage}
 \] 
\end{example}

\subsubsection{Complete $p$-graphs and $p$-subgraphs}
We will now exhibit a criterion to check if a $p$-RAAG associated to a complete $p$-graph is non-degenerate (or equivalently uniform). 
%

\begin{proposition}\label{lem:unif}
\begin{enumerate}[(a)]
\item Let $G=\langle x_1,\ldots,x_d\rangle$ be a uniform pro-$p$ group with $d = d(G)$  and assume that for all $1\le i<j\le d$ there exist $\a_{ij}, \beta_{ij} \in p\Z_p$ such that $[x_i,x_j] = x_i^{\a_{ij}} x_j^{\beta_{ij}}$. 
Then there exists a complete $p$-graph $\Gamma$ such that $G=G_\Gamma$.
\item Let $\Gamma$ be a complete $p$-graph and let $G_\Gamma= \langle x_1,\ldots,x_d \rangle$ be its associated generalised $p$-RAAG. Then $G_\Gamma$  is uniform if and only if every triple of generators $x_r,x_s,x_t$ in $G_\Gamma$ generates a torsion free pro-$p$ group (which must be uniform of dimension 3). In particular, $G_\Gamma$ is quadratic if and only if it is torsion free, and therefore uniform.
\end{enumerate}
\end{proposition}

\begin{proof} 
If $G$ is a uniform pro-$p$ group as in $(a)$, then there is a presentation of the required form over a complete $p$-graph $\Gamma$
by \cite[Proposition~4.32]{ddms:padic}. This proves $(a)$.

Let $\Gamma$ be a complete $p$-graph.
Modulo reversing some arrows, the $p$-RAAG $G_\Gamma$ has a presentation
\begin{equation}\label{eq:presuniform}
  G_\Gamma = \pres{x_1,\ldots,x_d}{[x_i,x_j] = x_i^{\a_{ij}} x_j^{\beta_{ij}},\ 1\le i<j\le d}.
\end{equation}
Clearly $G_\Gamma$ is a powerful group. Let $H_{r,s,t}$ be the subgroup of  $G_\Gamma$  generated by a triple of generators
$x_r,x_s,x_t$ in  $G_\Gamma$ . If  $G_\Gamma$ is uniform, then it is torsion free and therefore  $H_{r,s,t}$ is  torsion free as well. Now suppose that  $H_{r,s,t}$ is torsion free for every triple of generators $x_r,x_s,x_t$ in $G_\Gamma$.  Then $H_{r,s,t}$ is  uniform of dimension $3$.
Hence the $\mathbb{Z}_p$-Lie algebra $L_{H_{r,s,t}}$ associated to $H_{r,s,t}$ has the presentation
\begin{equation}\label{eq:presliealg}
  L_{H_{r,s,t}} = \left\langle X_r, X_s, X_t  \ \left\vert \ \begin{matrix} [X_r,X_s] = \alpha_{r,s}' X_r + \beta_{r,s}' X_s,\\ [X_s,X_t] = \alpha_{s,t}' X_s + \beta_{s,t}' X_t,\\ [X_t,X_r] = \alpha_{t,r}' X_t + \beta_{t,r}' X_r \end{matrix}  \right. \right\rangle
\end{equation}
for some $\alpha_{ij}', \beta_{ij}'\in p\mathbb{Z}_p$. In fact, \eqref{eq:presliealg} follows from \eqref{eq:presuniform} by direct computation using \eqref{eq:Liebracket} and the hypothesis. 
Define the powerful $\mathbb{Z}_p$-Lie algebra $L$ given by the presentation
\[
  L = \pres{X_1, \ldots, X_d}{[X_i,X_j] = \a_{ij}' X_i + \beta_{ij}' X_j,\ 1\le i<j\le d}.
\]
It is easy to see that the free $ \Z_p$-module $L$ with basis $X_1,\ldots,X_d$ has a Lie algebra structure with the given value of the Lie bracket on the generators; it suffices to check the Jacobi identity for triples of generators. Consider the uniform pro-$p$ group $G_L$ associated to $L$
via the Lazard correspondence using the map $\exp: L\to G_L$. 
Using the properties of the commutator Campbell-Hausdorff formula $\Psi: L\to G_L$
(\cite[Lem.~7.12(iii)]{ddms:padic}), one can show that
$$\exp ([X_r,X_s]) = [\exp(X_r), \exp(X_s)]=  \exp(X_r)^{\a_{rs}} \exp(X_s)^{\b_{rs}}.$$
Thus the map $G_L \to G$ defined by $\exp(X_r) \mapsto x_r$ yields a surjective homomorphism,
which induces the long exact sequence 
\[
 \xymatrix@C=1.2truecm{ 0\ar[r] & H^1(G,\F_p)\ar[r]^-{\inf_{G_L,N}^1}& H^1(G_L,\F_p)\ar[r]^-{\res_{G_L,N}^1} 
& H^1(N,\F_p)^G \ar`r[d]`[l] `[dlll] `[dll] [dll]   \\
 & H^{2}(G,\F_p)\ar[r]^-{\inf_{G_L,N}^2} & H^{2}(G_L,\F_p) &  &}
\]
where $N$ denotes the kernel of $G_L\to G$.
The map $\inf_{G_L,N}^1$ is an isomorphism because we may identify the bases of $G_L$ and of $G$.
Moreover, also the map $\inf_{G_L,N}^2$ is an isomorphism, since Proposition~\ref{prop:uniform} yields
$H^2(G,\F_p)\cong H^2(G_L,\F_p)\cong\Lambda_2H^{1}(G,\F_p)$.
Therefore, $H^1(N,\F_p)^G$ is trivial, so that $N=\kernel(G_L\to G)$ is trivial too.
\end{proof}

Proposition~\ref{lem:unif} also gives a handy criterion to check if a complete $p$-RAAG is non-degenerate. In fact, we believe that deciding whether a generalised $p$-RAAG is degenerate boils down to the same question for $p$-subgraphs that are triangles. This can be a very difficult task. We attempt to give some partial answers in Section~\ref{sec:triang} where we will study the groups arising from $p$-graphs of the form
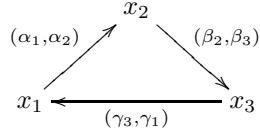
\[ 
  \xymatrix{& x_2 \ar[dr]^{(\b_2,\b_3)} & \\ x_1 \ar[ur]^{(\a_1,\a_2)}&  & x_3 \ar[ll]^{(\c_3,\c_1)}} 
\]
We will call these groups \emph{triangle $p$-RAAGs}.

Furthermore, it is hard to decide whether the $p$-RAAG associated to a $p$-subgraph $\Gamma_1\subset \Gamma$ embeds into the $p$-RAAG associated to $\Gamma$. This can be addressed for complete $p$-subgraphs of non-degenerate $p$-graphs.

\begin{lemma}\label{lemma:complete in nondeg}
 Let $\Gamma=((\sV,\sE),f)$ be a $p$-graph such that the associated $p$-RAAG $G_\Gamma$ is non-degenerate,
 and let $\Delta=((\sV_\Delta,\sE_\Delta),f_\Delta)$ be a complete $p$-subgraph of $\Gamma$.
Then the $p$-RAAG $G_\Delta$ is uniform, and it embeds in $G_\Gamma$. 
\end{lemma}

\begin{proof}
 Let $\tilde\Gamma$ be a complete $p$-graph which completes $\Gamma$.
 Then $\Delta$ is a $p$-subgraph of $\tilde\Gamma$, and one has the morphisms of pro-$p$ groups
 \[
\xymatrix{ G_\Delta\ar[r]^-{\phi_\Delta} & G_\Gamma\ar[r]^-{\phi_{\tilde\Gamma}} & G_{\tilde\Gamma}} , 
\]
with $G_{\tilde\Gamma}$ uniform.
Set $\psi=\phi_{\tilde\Gamma}\circ\phi_\Delta$.
Since $\Delta$ is complete, $G_\Delta$ is powerful, and therefore also $\image(\psi)$
is powerful.
Moreover, $G_{\tilde\Gamma}$ is torsion-free, thus also $\image(\psi)$ is torsion-free, 
and hence uniform.
Finally, both $G_\Delta$ and $\image(\psi)$ are minimally generated by $\sV_\Delta$.
Hence, $\psi\colon G_\Delta\to\image(\psi)$ is an isomorphism.
\end{proof}

In the rest of this section we will try to apply Theorem~\ref{thm:cohomology amalgam} to show that
several families of $p$-RAAGs yield quadratic pro-$p$ groups.
We have seen in Remark~\ref{rmk:hypotheses pRAAGs} that $p$-RAAGs automatically satisfy the cohomological hypotheses
of Theorem~\ref{thm:cohomology amalgam}. It turns out that the main obstruction to apply the theorem
is the fact that we do not have a general criterion to decide whether a certain $p$-RAAG $G_\Gamma$ is
a \emph{proper} amalgam of two $p$-RAAGs associated to full subgraphs of $\Gamma$.
In the next subsection, we will add two novel criteria for an amalgam of pro-$p$ groups to be proper.

\subsection{Proper amalgams of \texorpdfstring{$p$}{p}-RAAGs}

We will show below that the amalgam of two uniform pro-$p$ groups over a uniform subgroup $H$ is always proper, provided that the generators of $H$ are part of the minimal generating sets of both groups. This adds a new criterion to the known criteria from \cite[\S~9.2]{ribzal:book}. We will add yet another new criterion for properness in later in this Section.

\begin{proposition}\label{prop:amalgunif}
Let $1\le d\le k \le n$. Let $G_1=\langle x_1,\ldots,x_k\rangle$ and $G_2= \langle x_d\ldots,x_n \rangle$ be uniform pro-$p$ groups with the isomorphic closed uniform subgroup $H=\langle x_d,\ldots,x_k \rangle$. Then the amalgamated free pro-$p$ product $G=G_1 \amalg_H G_2$ is proper.
\begin{proof}
First note that it is sufficient to show that $H^{p^n} = H\cap G_i^{p^n}$ for every $n$ and $i=1,2$. In fact, we can then apply \cite[Thm.~9.2.4]{ribzal:book} and we are done.

By symmetry, it is sufficient to show the property for $i=1$. It is clear that $H^{p^n} \le H\cap G_1^{p^n}$ for all $n$. By \cite[Thm.~2.7]{ddms:padic}, we have that $G_1^{p^n} = \langle x_1^{p^n},\ldots,x_k^{p^n} \rangle$ and $H^{p^n} = \langle x_d^{p^n},\ldots,x_k^{p^n} \rangle$. Suppose by contraddiction that there is $g\in (H\cap G_1^{p^n}) \smallsetminus H^{p^n}$. Then $$ g= x_1^{a_1 p^n} \ldots x_{d-1}^{a_{d-1} p^n} x_d^{a_d p^n} \ldots x_k^{a_k p^n}$$ for some $a_i,b_j \in \mathbb{Z}_p$. Since $x_d^{a_d p^n} \ldots x_k^{a_k p^n} \in H$, we also have $x_1^{a_1 p^n} \ldots x_{d-1}^{a_{d-1} p^n} \in H$. Thus there exist $c_d,\ldots,c_k\in \mathbb{Z}_p$ such that 
$$x_1^{a_1 p^n} \ldots x_{d-1}^{a_{d-1} p^n}\cdot x_d^{c_d p^n} \ldots x_{k}^{c_{k} p^n}  = 1.$$ Since $G_1$ is uniform, we can conclude that $$a_1=\ldots =a_{d-1} = c_d =\ldots = c_k=0,$$ and hence $g\in H^{p^n}$, which yields a contradiction.
\end{proof}
\end{proposition}

We now need an auxilliary lemma on free products of pro-$p$ groups and their Frattini subgroups.
Recall that the Frattini series of a pro-$p$ group $G$ is defined inductively by $\Phi^1(G)=G$ and $\Phi^{n+1}(G) = \Phi^n(G)^p [\Phi^n(G),\Phi^n(G)]$.

\begin{lemma}\label{lem:frattini}
Let $H$ and $K$ be pro-$p$ groups and let $G= H\ast K$. Then $$\Phi^n(H) = H\cap \Phi^n(G) \text{\quad for all $n$}.$$
\begin{proof}
Clearly $\Phi^n(H) \le H\cap \Phi^n(G)$. For the other inclusion, we first observe that there exists a retraction $r:G\to H$, i.e.\ a continuous homomorphism such that $r_{\vert H} = \mathrm{id}_H$, given by $h\mapsto h$ and $k\mapsto 1$ for $h\in H$ and $k\in K$. Hence $r(\Phi^n(G)) \le \Phi^n(r(G))$ and \begin{multline*} \Phi^n(G)\cap H = r(\Phi^n(G)\cap H) \le r(\Phi^n(G)) \cap r(H) \le \\ \le \Phi^n(r(G)) \cap H = \Phi^n(H) \cap H = \Phi^n(H).\end{multline*}
\end{proof}
\end{lemma}

We are almost ready to state our criterion for an amalgam of $p$-RAAGs to be proper. First one more proposition which might be of independent interest.

\begin{proposition}\label{prop:frattiniunifsubgroup}
Let $G=G_\Gamma=\langle x_1,\ldots,x_d\rangle $ be a non-degenerate $p$-RAAG. Suppose that the subgroup $H=\langle x_k\ldots, x_d\rangle$ is uniform for some $1 \le k\le d$. Then $$ \Phi^n(H) = H\cap \Phi^n(G) $$ for every $n\in \mathbb{N}$.
\begin{proof}
Clearly $\Phi^n(H) \le H\cap \Phi^n(G)$. Consider the free pro-$p$ group $F=F(x_1,\ldots,x_{k-1})$ generated by $x_1,\ldots,x_{k-1}$ and the canonical projection $\varphi: F\ast H \to G$ given by $x_i\mapsto x_i$ for $i=1,\ldots,d$. 
Also, denote by $\widetilde{G}$ the uniform quotient of $G$ from Definition~\ref{defn:nondeg} and by $\pi: G\to \widetilde{G}$ the associated projection.
Then the map $(\pi \circ \varphi)_{\vert H}$ is an isomorphism, since $H$ is uniform by hypothesis. We will use the same symbol for the three different copies of $H$ to make the notation lighter.

Suppose by contradiction that there is $x\in (H\cap \Phi^n(G)) \smallsetminus \Phi^n(H)$. Then, by Lemma~\ref{lem:frattini}, $\pi(x)\in H\cap \Phi^n(G) = \Phi^n(H)$. Moreover, there exists $y\in \Phi^n(H)= H\cap \Phi^n(F\ast H)$ such that $\pi(\varphi(y)) = \pi(x)$. Since $y\in H$ and $\pi_{\vert H}$ is also an isomorphism, we deduce that $\varphi(y)=x$. Thus $x\in \varphi(\Phi^n(H)) \le \Phi^n(H)$,  which yields a contradiction.   
\end{proof}
\end{proposition}

Let us go back to the task at hand, that is showing the properness of amalgams over uniform subgroups in certain non-degenerate cases.

\begin{proposition}\label{lem:amalgunifsubgroup}
 Let $G_1= G_{\Gamma_1}$ and $G_2 = G_{\Gamma_2}$ be non-degenerate $p$-RAAGs with underlying $p$-graphs $\Gamma_1$ and $\Gamma_2$, respectively. Let $\Gamma'$ be a common isomorphic complete $p$-subgraph of $\Gamma_1$ and $\Gamma_2$ and let $H=G_{\Gamma'}$. Then the amalgamated product $G_1 \amalg_H G_2$ is proper.
\begin{proof}
By Proposition~\ref{prop:frattiniunifsubgroup}, we have that $H\cap \Phi^n(G_i)=\Phi^n(H)$ for $i=1,2$ and every $n\in \mathbb{N}$. Now the result follows from \cite[Thm.~9.2.4]{ribzal:book}, where we can take $U_{i n} = \Phi^{n}(G_i)$.
\end{proof}
\end{proposition}

This generalises the well-known fact that amalgams over pro-cyclic subgroups are always proper.

\begin{example}
Let $K_1$ and $K_2$ be non-degenerate $p$-RAAGs and let $H$ be a uniform pro-$p$ group. Consider the groups $G_1=K_1\times H$ and $G_2= H\times K_2$. Then $G=G_1\amalg_H G_2$ is a proper amalgam  by \cite[Ex.~9.2.6(c)]{ribzal:book}. Note that, if $H$ is $p$-RAAG, then properness follows also from Proposition~\ref{lem:amalgunifsubgroup}.

Moreover, if both $K_1$ and $K_2$ are quadratic, by Proposition~\ref{prop:cohomology freeproduct}
and Proposition~\ref{prop:uniform} also $G_1$ and $G_2$ are quadratic, and for both $i=1,2$ one has 
\[\begin{split}
   H^1(G_i,\F_p) &= H^1(K_i,\F_p)\oplus H^1(H,\F_p),\\
   H^2(G_1,\F_p) &= H^2(K_i,\F_p)\oplus (H^1(K_i,\F_p)\wedge H^1(H,\F_p))\oplus H^2(H,\F_p),
\end{split} \]
and the restriction maps $\res_{G_i,H}^1$ and $\res_{G_i,H}^2$ are the projections onto the second summand
of $H^1(G_i,\F_p)$ and the third summand of $H^2(G_i,\F_p)$ respectively.
Hence, all the hypothesis of Theorem~\ref{thm:cohomology amalgam} are satisfied
and $G$ is quadratic. 
\end{example}

\subsection{Some quadratic triangle-ful \texorpdfstring{$p$}{p}-RAAGs}\label{sec:sometrianglefull}

Next we will produce several examples of triangle-ful $p$-RAAGs that are quadratic. First of all we remark that all ``small'' non-degenerate $p$-RAAGs are quadratic.

\begin{lemma}\label{lem:smallpRRAG}
Let $\Gamma=((\sV,\sE),f)$ be a $p$-graph.
\begin{enumerate}
\item If $\norm{\sV}\le 3$ and $\norm{\sE}\le 2$, then $G_\Gamma$ is quadratic. 
\item Suppose that $\norm{\sV}= 3$ and $\norm{\sE} = 3$ and that $G_\Gamma$ is non-degenerate. Then $G_\Gamma$ is quadratic.
\end{enumerate}
\begin{proof}
In the first case, $G_\Gamma$ is either a free pro-$p$ group, a Demushkin group or an amalgamated free product of a free pro-$p$ group and a Demushkin group over a pro-cyclic subgroup. The result follows from \cite[Thm.~3.2]{ribes:amalg} and Theorem~\ref{thm:cohomology amalgam}.

In the second case, $G_\Gamma$ is a uniform pro-$p$ group by Proposition~\ref{lem:unif}. The result follows from Proposition~\ref{prop:uniform}.
\end{proof}
\end{lemma}

\begin{remark}\label{rmk:newquadpraags}
 We now exhibit some operations that, starting from quadratic $p$-RAAGs, produce new quadratic $p$-RAAGs. We point out that, in all the following examples, we can apply Theorem~\ref{thm:cohomology amalgam} because of Remark~\ref{rmk:hypotheses pRAAGs}.
\begin{enumerate}[(a)]
\item \textbf{Disjoint union.} If $\Gamma$ is the union of two disjoint subgraphs $\Gamma_1$ and $\Gamma_2$, then $G_\Gamma$ is isomorphic to $G_{\Gamma_1}\ast G_{\Gamma_2}$. Therefore, if $\Gamma_1$ and $\Gamma_2$ are quadratic $p$-RAAGs, then $G_\Gamma$ is a quadratic $p$-RAAG as well. Moreover, free products of quadratic $p$-RAAGs are quadratic $p$-RAAGs.
\item \textbf{Mirroring.} Let $G_\Gamma$ be a quadratic $p$-RAAG. Then the amalgamated product $G =G_\Gamma \amalg_{G_{\Gamma'}} G_\Gamma$ of $G_\Gamma$ with itself over  $G_{\Gamma'}$ (identified via the identity map) , where $\Gamma'$ is any full subgraph of $\Gamma$,  is proper by \cite[Ex.~9.2.6(a)]{ribzal:book}. Hence such a $G$ is quadratic.
\item \textbf{Amalgam over uniform subgroups.} The $p$-RAAGs obtained from Proposition~\ref{prop:amalgunif} and Proposition~\ref{lem:amalgunifsubgroup}.
\item \textbf{RAAGs.} Pro-$p$ completions of abstract RAAGs can be written as a series of proper HNN-extensions, hence they are quadratic. This fact was already proved by Riley and Weigel (unpublished).
\end{enumerate}
\end{remark}

\begin{example}
Let $\Gamma$ be a $p$-graph obtained from a complete $p$-graph by removing one edge. Using Proposition~\ref{prop:amalgunif}, we can show that the $p$-RAAG $G_\Gamma$ is quadratic. One can make similar considerations using Proposition~\ref{lem:amalgunifsubgroup}.
\end{example}

Next we will show that another special class of triangle-ful $p$-graphs that yields several new quadratic $p$-RAAGs.

\begin{definition}\label{def:chordal}
 A graph $\calG$ is \emph{chordal} (or \emph{triangulated}) if it contains no circuits other than triangles as full subgraphs.
\end{definition}

Chordal graphs are characterized by the following property (cf.\ \cite[Prop.~5.5.1]{graphbook}).

\begin{proposition}\label{prop:chordalgraph}
 A graph is chordal if and only if it can be constructed
recursively by pasting along complete subgraphs, starting from
complete graphs.
\end{proposition}


Now we will prove Theorem~\ref{thm:chordal}.

\begin{proof}[\textbf{Proof of Theorem~\ref{thm:chordal}}]
We proceed by induction on the size of $\Gamma$.
Clearly, the theorem holds if $\Gamma$ has only one vertex.
Therefore, we can assume that $\Gamma$ is a graph with more than one vertex.
If $\Gamma$ is complete, then $G= G_{\Gamma}$ is a uniform pro-$p$ group, and therefore quadratic (cf. Proposition~\ref{lem:unif}--(b)).
Otherwise, by Proposition~\ref{prop:chordalgraph}, there are proper full subgraphs $\Gamma_1,\Gamma_2,\Delta$ of $\Gamma$, with $\Delta$ complete such that $\Gamma$ is obtained by pasting together $\Gamma_1$ and $\Gamma_2$, and $\Delta= \Gamma_1 \cap \Gamma_2$.
Thus 
\begin{equation}\label{eq:amalg chordal}
 G_{\Gamma} = G_{\Gamma_1} \amalg_{G_\Delta} G_{\Gamma_2},
\end{equation}
where $G_{\Delta}$ is a uniform pro-$p$ group, since $G_\Gamma$ is non-degenerate. 
Clearly, also $G_{\Gamma_1}$ and $G_{\Gamma_2}$ are non-degenerate, and $\Gamma_1,\Gamma_2$ are chordal.
Thus, by Proposition~\ref{lem:amalgunifsubgroup}, the amalgam~\eqref{eq:amalg chordal} is proper.
By induction, $G_{\Gamma_1}$ and $G_{\Gamma_2}$ are quadratic pro-$p$ groups. 
Hence, by Theorem~B, $G$ is a quadratic pro-$p$ group.
\end{proof}

We believe that all non-degenerate $p$-RAAGs are quadratic, but this might be very hard to prove given the limited knowledge of properness of amalgamated free products in the category of pro-$p$ groups.

 As evidence of the power of our methods, we remark that every $p$-graph on at most $4$ vertices always yields a quadratic $p$-RAAG.
Moreover, all $p$-graphs on $5$ vertices but 
those with underlying graph $\mathcal{H}$ as in Figure~\ref{fig:2} can be handled with the same methods. 
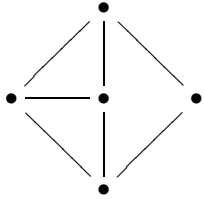
\begin{figure}[h]
\[ 
  \xymatrix{ & \bullet \ar@{-}[dr] \ar@{-}[dl] \ar@{-}[d] & \\   \bullet \ar@{-}[dr] \ar@{-}[r] & \bullet \ar@{-}[d] & \bullet  \ar@{-}[dl] \\  & \bullet & \\ & & } 
\vspace{-1cm}
\]
\caption{The graph $\mathcal{H}$}\label{fig:2}
\end{figure}

If the $p$-labels on $\mathcal{H}$ are ``symmetric along the horizontal axis'', this $p$-graph can be handled via mirroring (i.e.\ Remark~\ref{rmk:newquadpraags}(b) above).

On the other hand, there are instances where we cannot decide, in general, whether a non-degenerate $p$-graph
always yields a quadratic group.

\subsection{Generalised \texorpdfstring{$p$}{p}-RAAGs and Galois groups}\label{ssec:Galois}

Throughout this subsection, $K$ denotes a field containing a root of unity of order $p$, and also $\sqrt{-1}$ if 
$p=2$.
The following result shows that for some fields $K$, the maximal pro-$p$ Galois group $G_{K}(p)$
is a generalised $p$-RAAG.

\begin{proposition}\label{prop:solvable galois}
Let $G$ be a finitely generated solvable pro-$p$ group which occurs as $G_{K}(p)$ for some field $K$.
Then $G\cong G_\Gamma$ for a complete $p$-graph $\Gamma$.
\end{proposition}

\begin{proof}
By \cite[Cor.~4.9]{cq:bk}, if $G$ is solvable then it is uniform, and moreover every 2-generated subgroup of $G$
is again uniform --- i.e. it is a 2-generated Demushkin group.
Thus, given a basis $\mathcal{X}=\{x_1,\ldots,x_d\}$ of $G$, one has $[x_i,x_j]\in\langle x_i,x_j\rangle$
and Proposition~\ref{lem:unif}  yields the claim.
\end{proof}

Clearly, the Bloch-Kato conjecture implies that being quadratic is a necessary condition
for a generalised $p$-RAAG to occur as maximal pro-$p$ Galois group for some field $K$.
It is natural asking whether there are further conditions that a generalised $p$-RAAG must fulfill in order
to occur as maximal pro-$p$ Galois group of a field $K$.
For example, one has the following obstruction.

\begin{example}\label{example:galois square}
 Let $\Gamma=(\calG,f)$ be the square $p$-graph with labels all equal to $(0,0)$.
Then, by \cite[Thm.~5.6]{cq:bk}, the $p$-RAAG $G_\Gamma$ can not be realized as $G_{K}(p)$ for any field $K$. 
\end{example}

We will see shortly another necessary condition about how the edges of a $p$-graph $\Gamma=(\calG,f)$ and their labels patch together.
Let $(x_i,x_j)$ be an edge of $\calG$ and let $H_{ij}=\langle x_i,x_j\rangle$ be the subgroup
of the $p$-RAAG $G_\Gamma$ generated by $x_i$ and $x_j$.
By Example~\ref{ex:pRAAGs} and Lemma~\ref{lem:isom}, if $H_{ij}$ is uniform then it is a 2-generated
Demushkin group and there exist $u_{ij},w_{ij}\in H_{ij}$ such that 
\begin{equation}\label{eq:pres cyc}
 H_{ij}=\langle u_{ij},w_{ij}\mid u_{ij}w_{ij}u_{ij}^{-1}=w_{ij}^{1+\lambda_{ij}}\rangle,
\end{equation}
for some $\lambda_{ij}\in p\Z_p$.
In particular, the structure of $H_{ij}$ induces a homomorphism of pro-$p$ groups 
\begin{equation}\label{eq:cyc morph}
 \theta_{ij}\colon H_{ij}\to1+p\Z_p, \qquad \theta_{ij}(u_{ij})=1+\lambda_{ij},\ \theta_{ij}(w_{ij})=1.
\end{equation}
Recall that $1+p\Z_p=\{1+p\lambda\mid\lambda\in\Z_p\}$, which is isomorphic to $\Z_p$ if $p\neq2$,
and to $C_2\oplus\Z_2$ if $p=2$.

\begin{definition}\label{defi:kummerian graph}
A $p$-graph $\Gamma=(\calG,f)$, with undelying graph $\calG=(\sV,\sE)$, is called \emph{cyclotomic} if:
\begin{itemize}
\item[(a)] for every edge $(x_i,x_j)\in\sE$, the subgroup $H_{ij}$ is uniform, and
 \item[(b)] for all edges $(x_i,x_j),(x_j,x_h)\in\sE$ we have $\theta_{ij}(x_j)=\theta_{jh}(x_j)$ (cf.\ \eqref{eq:cyc morph}).
\end{itemize}
\end{definition}

Namely, in a cyclotomic $p$-graph $\Gamma=(\calG,f)$
the homomorphisms $\theta_{ij}$ induced by the edges of $\calG$, which are 2-generated Demushkin groups,
agree on common vertices.
The following shows that being cyclotomic is a necessary condition for a $p$-graph in order
to give rise to a generalised $p$-RAAG which is a maximal pro-$p$ Galois group.

\begin{theorem}\label{thm:galois pRAAGs}
Let $K$ be a field containing a root of unity of order $p$ and let $\Gamma=(\calG,f)$ be a $p$-graph with underlying graph $\calG=(\sV,\sE)$. Suppose that $G_\Gamma\cong G_{K}(p)$ for some field $K$.
Then $\Gamma$ is cyclotomic. 
\end{theorem}

\begin{proof}
Let $$\theta\colon G_\Gamma\longrightarrow1+p\Z_p$$ be the $p$th cyclotomic character induced by the action of $G_{K}(p)$ on the 
roots of unity of order a power of $p$ lying in the maximal pro-$p$ extension $K(p)$ of $K$
(cf., e.g., \cite[Def.~7.3.6]{nsw:cohn}).
For an edge $(x_i,x_j)\in\sE$, the subgroup $H_{ij}$ is the maximal pro-$p$ Galois group
of the subextension $K(p)/K(p)^{H_{ij}}$.

Since $H_{ij}$ is not free, then $H_{ij}$ is uniform by \cite[Thm.~4.6]{cq:bk},
with $\theta_{ij}\colon H_{ij}\to1+p\Z_p$ the cyclotomic character of the extension $K(p)/K(p)^{H_{ij}}$,
which coincides with the restriction $\theta\vert_{H_{ij}}$.
Therefore, for all edges $(x_i,x_j),(x_j,x_h)\in\sE$, one has 
$ \theta_{ij}(x_j)=\theta(x_j)=\theta_{jh}(x_j)$.
\end{proof}

The following result shows that cyclotomic $p$-graphs are also a good source of non-de\-ge\-ne\-ra\-te $p$-RAAGs.

\begin{proposition}\label{prop:cyclotomic graph}
Let $\Gamma=(\calG,f)$ be a cyclotomic $p$-graph with underlying graph $\calG=(\sV,\sE)$.
Then the $p$-RAAG $G_{\Gamma}$ is non-degenerate.
\end{proposition}

\begin{proof}
By Remark~\ref{rmk:newquadpraags}(a), we may assume that $\calG$ is connected, so that every vertex $x_i\in\sV$ belongs to some edge.
Let $\theta\colon G_\Gamma\to1+p\Z_p$ be the homomorphism induced by the homomorphisms $\theta_{ij}$ for every edge
of $\Gamma$ --- since $\Gamma$ is cyclotomic, $\theta$ is well defined.
Moreover, let \eqref{eq:presentation} be the minimal presentation of $G_\Gamma$ induced by $\Gamma$ and let 
$\hat\theta\colon F\to1+p\Z_p$ be the composition of the projection $F\to G_\Gamma$ with $\theta$.

Consider the following normal subgroups
\[\begin{split}
   K(G_\Gamma)&=\left\langle\left. y^{-\theta(x)}xyx^{-1}\right\vert x\in G_\Gamma,y\in\kernel(\theta) \right\rangle \lhd G_\Gamma,\\
   K(F)&=\left\langle\left. y^{-\hat\theta(x)}xyx^{-1}\right\vert x\in F,y\in\kernel(\hat\theta) \right\rangle 	\lhd F
 \end{split}\]
 (cf.\ \cite[\S~3]{eq:kummer}).
Note that $K(G_\Gamma)\le\Phi(G_\Gamma)\cap\kernel(\theta)$ and $K(F)\le\Phi(F)\cap\kernel(\hat\theta)$.
 By \eqref{eq:pres cyc}, $R$ is generated as normal subgroup of $F$ by the relations 
$$w_{ij}^{-\hat\theta(u_{ij})}u_{ij}w_{ij}u_{ij}^{-1},\qquad \text{for }(x_i,x_j)\in\sE(\calG),$$
and consequently $R\le K(F)$.
Therefore, \cite[Thm.~5.6]{eq:kummer} implies that the quotient $\bar G_\Gamma:=G_\Gamma/K(G_\Gamma)$ is torsion-free
and it splits as semi-direct product
\[
 \bar G_\Gamma\cong\Z_p^m\rtimes (G_\Gamma/\kernel(\theta)),\qquad \text{for some }m\geq0,\]
with action $\bar xz\bar x^{-1}= z^{\bar\theta(\bar x)}$ for all $z\in\Z_p^m$ and $\bar x\in\bar G_\Gamma$, 
where $\bar\theta\colon \bar G_\Gamma\to1+p\Z_p$ is the morphism induced by $\theta$
(namely, the pro-$p$ group $G_\Gamma$ endowed with the morphism $\theta$ is ``Kummerian'', following the language of \cite{eq:kummer}).
In particular, $[x,y]\in\langle x,y\rangle\le \bar G_\Gamma$.
Therefore, by Proposition~\ref{lem:unif}~(a) there exists a complete $p$-graph $\tilde \Gamma$ such that 
$\bar G_{\Gamma}\cong G_{\tilde\Gamma}$.
Since $\sV(\tilde\Gamma)=\sV(\Gamma)$, $\tilde\Gamma$ is a completion of $\Gamma$.
\end{proof}

The class of \emph{Koszul graded algebras} is a particular class of quadratic algebras,
singled out by Priddy in \cite{priddy} --- the definition of Koszul graded algebra
is highly technical; we refer to \cite[Ch.~2]{pp:quad} and to \cite[\S~2]{MPQT}.
Recently, Koszul graded algebras became of great interest in the context of Galois cohomology
(see, e.g., \cite{pos:K,posi:koszul,MPQT}).
In particular, Positselski conjectured in \cite{posi:koszul} that the cohomology algebra $H^\bullet(G_{K}(p),\F_p)$ is Koszul, if $G_{K}(p)$ is finitely generated.
Moreover, Weigel conjectured in \cite{weigel:collection} that the graded group algebra
\[
 \mathrm{gr}(\F_pG_{K}(p))=\bigoplus_{n\geq0}I^n/I^{n+1},\qquad I^0=\F_pG_{K}(p), 
\]
where $I$ denotes the augmentation ideal of the group algebra $\F_pG_{K}(p)$ (cf.\ \cite[\S~3.2]{MPQT}),
is also a Koszul graded algebra.
Usually it is quite hard to check whether a graded algebra is Koszul. 
Nonetheless, in the setting of generalised $p$-RAAGs we can easily deduce the following.

\begin{coro} Let $\Gamma=(\calG,f)$ be a $p$-graph and let $G_\Gamma$ be the associated $p$-RAAG.
\begin{itemize}
\item[(i)] If $G_\Gamma$ is quadratic, then
the cohomology algebra $H^\bullet(G_\Gamma,\F_p)$ is Koszul.
\item[(ii)] If $\Gamma$ is triangle-free, then the graded group algebra $\mathrm{gr}(\F_pG_\Gamma)$ is Koszul.
\end{itemize}
\end{coro}

\begin{proof}
Statement (i) follows from Theorem~\ref{thm:pRAAGs quadratic} and \cite[\S~4.2.2]{weigel:koszul}.
Statement (ii) follows from Theorem~\ref{thm:pRAAGs mild} and \cite[Thm.~8.4]{MPQT}.
\end{proof}

Thus, generalised $p$-RAAGs provide a huge source of pro-$p$ groups for which Positselski's and Weigel's 
\emph{Koszulity conjectures} hold.
The above result raises the following question.

\begin{ques}
 Let $G_\Gamma$ be a quadratic $p$-RAAG with associated $p$-graph $\Gamma$.
Is the graded algebra $\mathrm{gr}(\F_p G_\Gamma)$ Koszul?
\end{ques}

\section{Triangle \texorpdfstring{$p$}{p}-RAAGs}\label{sec:triang}

In the following section we will slightly change the focus of our investigation. Until now we were mainly concerned with finding new examples of quadratic pro-$p$ groups in the family of $p$-RAAGs. Here we will mainly be concerned with the determination of the isomorphism classes of quadratic $p$-RAAGs arising from triagle $p$-graphs. These will be called \emph{triangle $p$-RAAGs}. 

\subsection{The Lazard correspondence}

Given a powerful pro-$p$ group $G$ and $n\in \mathbb{N}$, we have $P_{n+1}(G)=G^{p^n}=\{ x^{p^n} ~|~ x\in G \}$ (see \cite[Thm.~3.6]{ddms:padic}). Moreover, if $G$ is uniform, then the mapping $ x \mapsto x^{p^n}$ is a homeomorphism from $G$ onto $G^{p^n}$ (see \cite[Lem.~4.10]{ddms:padic}). This shows that each element $x\in G^{p^n}$ admits a unique $p^n$th root in $G$, which we denote by $x^{p^{-n}}$.

As in the case of pro-$p$ groups, a $\mathbb{Z}_p$-Lie algebra $L$ is called \emph{powerful} if $L\cong \mathbb{Z}_p^d$ for some $d>0$ as $\mathbb{Z}_p$-module and $(L,L)_{Lie}\subseteq pL$ if $p$ is odd, or $(L,L)_{Lie} \subseteq 4L$ if $p=2$.

 If $G$ is an analytic pro-$p$ group, then it has a characteristic open subgroup which is uniform.
 For every open uniform subgroup $H \leq G$, $\mathbb{Q}_p[H]$ can be made into a normed $\mathbb{Q}_p$-algebra,
call it $\mathit{A}$, and $log(H)$, considered as a subset of the completion $\hat{A}$ of $A$,
will have the structure of a Lie algebra over $\mathbb{Z}_p$. 
There is a different construction of an intrinsic Lie algebra over $\mathbb{Z}_p$ for uniform groups.
The uniform group $U$ and its Lie algebra over $\mathbb{Z}_p$, call it $L_U$, are identified as sets,
and the Lie operations are defined by
\begin{equation}\label{eq:Liesum}
g+h=\lim_{n \to \infty}(g^{p^n}h^{p^n})^{p^{-n}}
\end{equation}
and 
\begin{equation}\label{eq:Liebracket}
 (g,h)_{Lie}=\lim_{n \to \infty}[g^{p^n},h^{p^n}]^{p^{-2n}}=\lim_{n\to \infty} (g^{-p^n}h^{-p^n}g^{p^n}h^{p^n})^{p^{-2n}} .
\end{equation}

It turns out that $L_U$ is a powerful $\mathbb{Z}_p$-Lie algebra and it is isomorphic to the $\mathbb{Z}_p$-Lie algebra $log(U)$ (cf.\ \cite[Cor.~7.14]{ddms:padic}).

On the other hand, if $L$ is a powerful $\mathbb{Z}_p$-Lie algebra, then the Campbell-Hausdorff formula induces
a group structure on $L$; the resulting group is a uniform pro-$p$ group.
If this construction is applied to the $\mathbb{Z}_p$-Lie algebra $L_U$ associated to a uniform group $U$,
one recovers the original group.
Indeed, the assignment $U\mapsto L_U$ gives an equivalence between the category of uniform pro-$p$ groups
and the category of powerful $\mathbb{Z}_p$-Lie algebras (see \cite[Thm.~9.10]{ddms:padic}).

In light of Theorem~\ref{thm:pRAAGs mild} and Proposition~\ref{lem:unif}, the structure of quadratic $p$-RAAGs
associated to triangle $p$-graphs is of particular interest.
So we are interested in torsion-free pro-$p$ groups defined by presentations of the form 
\begin{equation}
 G_{A'} =\langle x_1,x_2,x_3 \mid [x_1,x_2]=x_1^{\a_1'} x_2^{\a_2'}, [x_2,x_3]=x_2^{\b_2'} x_3^{\b_3'}, [x_3,x_1]=x_3^{\c_3'} x_1^{\c_1'} \rangle
\end{equation}
with parameters $A'=(\a_1',\a_2', \b_2',\b_3', \c_1',\c_3') \in p\mathbb{Z}_p^6$. Since the torsion-free group $G_{A'}$ is uniform, we can associate to it the $\mathbb{Z}_p$-Lie lattice $L_{G_{A'}}$.  It follows from \eqref{eq:Liebracket} that $L_{G_{A'}}$ has a ($\mathbb{Z}_p$-Lie algebra) presentation of the form
\begin{equation}\label{eq:Lielattice}
 L_A = \left\langle  x_1,x_2,x_3\ \left\vert\ \begin{matrix}
                                        [x_1,x_2] = \a_1 x_1 + \a_2 x_2,\ [x_2,x_3] = \b_2 x_2 + \b_3 x_3, \\
                                        [x_3,x_1] = \c_1 x_1 + \c_3 x_3 
                                       \end{matrix} \right.  \right\rangle
\end{equation}
for some $\a_1,\a_2$, $\b_2,\b_3$, $\c_1,\c_3 \in p\mathbb{Z}_p$. 
Hence, to understand the structure of triangle $p$-RAAGs, it is important to determine all $3$-dimensional $\mathbb{Z}_p$-Lie lattices of type \eqref{eq:Lielattice}. This will be done in the next subsection.
  
\subsection{Triangle Lie algebras} 

\begin{lemma}\label{lem:system}
 Let $L$ be a free $\mathbb{Z}_p$-module with basis $\{x_1,x_2,x_3\}$ and let $\a_1,\a_2$, $\b_2,\b_3$, $\c_1,\c_3 \in p\mathbb{Z}_p$. Consider the bilinear function $[\cdot,\cdot] : L\times L \to L$ defined by $[x_i,x_i]=0$ for $i=1,2,3$ and 
 \[
   [x_1,x_2] = \a_1 x_1 + \a_2 x_2,\quad [x_2,x_3] = \b_2 x_2 + \b_3 x_3,\quad [x_3,x_1] = \c_1 x_1 + \c_3 x_3.
 \]
Then $[\cdot,\cdot]$ defines a bracket on $L$ if and only if the following system is satisfied:
\begin{equation}\label{eq:system}
  \begin{cases}
    \a_1 \b_2 - \c_1 \b_3 =0 \\
    \c_1 \a_2 - \b_2 \c_3 = 0 \\
    \a_1 \c_3 - \a_2 \b_3 = 0
  \end{cases} 
\end{equation}
\end{lemma}
\begin{proof}
 We only need to check that \eqref{eq:system} is satisfied if and only if the function $[\cdot,\cdot]$ satisfies the Jacobi identity. Now, a straight-forward computation yields the claim.
\end{proof}

\begin{proposition}\label{prop:families}
 Let $\a_1,\a_2$, $\b_2,\b_3$, $\c_1,\c_3 \in p\mathbb{Z}_p$ and consider the $\mathbb{Z}_p$-Lie lattice $L_A$ defined in \eqref{eq:Lielattice}. Then there exist $\eta,\rho,\mu,\lambda \in p\mathbb{Z}_p$, satisfying $\eta \rho - \mu \lambda =0$, such that $L_A$ is isomorphic to one of the following:
 \begin{itemize}
   \setlength\itemsep{1em}
  \item $L_1(\eta,\rho,\mu,\lambda) = \langle x, y, z \vert [x,y]= \eta y,\ [y,z]= \mu y,\ [z,x] = \lambda z + \rho x   \rangle.$  
  \item $L_2(\eta,\mu) = \langle x, y, z \vert [x,y]= 0,\ [y,z]= \eta y + \mu z,\ [z,x] = 0  \rangle.$
  \item $L_3(\eta,\mu) = \langle x, y, z \vert [x,y]= 0,\ [y,z]= \eta z,\ [z,x] = \mu z   \rangle.$
  \item For $\eta\mu\lambda \neq 0$, $$L_4(\eta,\mu,\lambda) = \langle x,y,z \vert [x,y] =\eta x + \mu y, [y,z] = \lambda y- \eta z, [z,x] = - \lambda x - \mu z \rangle.$$
  \item For $\eta\mu\lambda \neq 0$, $$L_\ast(\eta,\mu,\lambda) = \langle x,y,z \vert [x,y] =\eta x + \mu y, [y,z] = \lambda y+ \eta z, [z,x] = \lambda x + \mu z \rangle.$$
 \end{itemize}
 \end{proposition}
\begin{proof}
Since $L_A$ is a Lie lattice, the parameters have to satisfy \eqref{eq:system}, by Lemma~\ref{lem:system}. First suppose that $\a_1=0$. It is easy to check the statement in this case, see Figure~\ref{fig:1} below for a schematic treatment.
\begin{figure}[h]
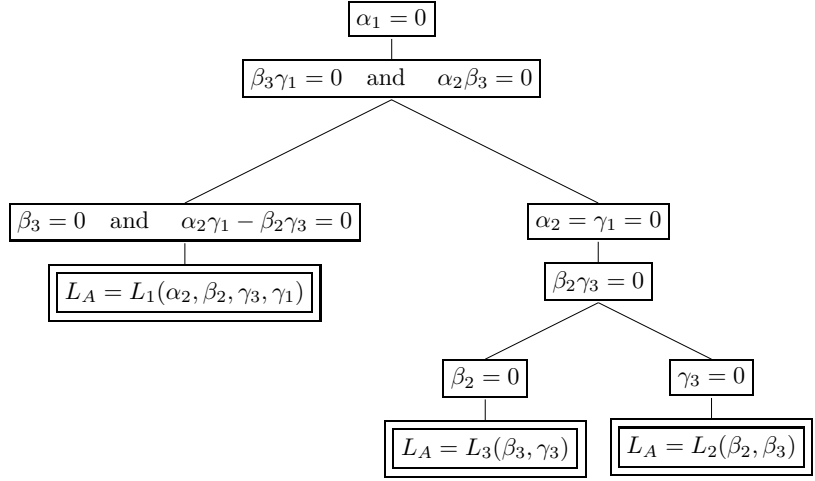

$$ 
\resizebox{0.80\linewidth}{!}{%
\Tree[.{\framebox{$\a_1 = 0$}} [ .{\framebox{$\b_3 \c_1= 0 \text{\quad and \quad} \a_2 \b_3 =0$}} [.{\framebox{$\b_3 = 0  \text{\quad and \quad} \a_2 \c_1 -\b_2 \c_3 =0$}} [.\framebox{{\framebox{$L_A= L_1(\a_2,\b_2,\c_3,\c_1)$}}} ]] [.\framebox{$\a_2 = \c_1 =0$}  [.{\framebox{$\b_2 \c_3 = 0$}}  [.{\framebox{$\b_2=0$}} \framebox{\framebox{$L_A=L_3(\b_3,\c_3)$}} ]   [.{\framebox{$\c_3 = 0$}}  \framebox{\framebox{$L_A = L_2(\b_2,\b_3) $}} ]]]]]}
 $$
 \caption{Proof of Proposition~\ref{prop:families} for $\a_1=0$.}\label{fig:1}
 \end{figure}

 If $\a_1\neq 0$ and some other coefficient is zero, we can use a change of basis to go back to the above case.
 
Suppose now that all coefficients are non-zero. Define the constants $$ q = \frac{\b_2}{\c_1} = \frac{\b_3}{\a_1},\quad r = \frac{\c_1}{\b_2} = \frac{\c_3}{\a_2},\quad s= \frac{\a_1}{\b_3} = \frac{a_2}{\c_3} .$$ Then $\a_1 \b_2 = \b_3 \c_1$, together with the definition of $r$, implies that $\a_1 = \b_3 r$. Since $\a_1= s \b_1$, we must have $r=s$. Furthermore $\b_2 \c_3 = \a_2 \c_1$, together with the definition of $q$, implies that $q \c_3 = \a_2$. Since $\a_2 = \c_3 s$, we must have $s=q$. So $q=r=s$.

Finally note that $\b_2 = q \c_1$ and $\c_1 = \b_2 r = \b_2 q$, therefore $q= \pm 1$. So either $q=r=s = 1$ or $q=r=s =-1$.
\end{proof}

The unusual numbering in the previous lemma will become clear after Lemma~\ref{lem:metabelianorSL21}. Recall that a non-zero element $z\in \mathbb{Z}_p$ can be written as formal power series $z=\sum_{n=N}^\infty a_n p^n$, for $a_n\in \{0,\ldots,p-1\}$ with $a_N\neq 0$, and we define $v_p(z)=N$.

\begin{lemma}\label{lem:xyz}
Each of the algebras of Proposition~\ref{prop:families} can be written in the form 
\[
  L_{\a,\b,\c}^{\xi_1,\xi_2,\xi_3} = \pres{x,y,z}{[x,y]= \a \xi_1 , [y,z] = \b \xi_2 , [z,x]= \c \xi_3}
\]
where $\a,\b,\c \in p\mathbb{Z}_p$, and $\xi_1\in \{0,x,y\}$, $\xi_2 \in \{0,y,z\}$ and $\xi_3 \in \{0,z,x,y\}$. Moreover, we have the following isomorphisms of Lie algebras
\[
   L_1(\eta,\rho,\mu,\lambda) \cong L_{0,\lambda,\eta}^{0,x,y},\  L_2(\eta,\mu) \cong L_{0,\mu,\eta}^{0,z,z},\ L_3(\eta,\mu) \cong L_{0,\eta,\mu}^{0,z,z}
\]
\[
   L_4(\eta,\mu,\lambda) \cong L_{0,\eta,-\eta}^{0,y,x} \text{ and }  L_\ast(\eta,\mu,\lambda) \cong L_{\eta,\eta,-2\rho}^{x,z,y}.
\]
\begin{proof}
The isomorphisms above can again be obtained by base-change. We will spell out the details in a few cases for the sake of clarity. 
 
 Clearly $L_3(\eta,\mu)\cong L_{0,\eta,\mu}^{0,z,z}$. 
 Consider the algebra $L_2(\eta,\mu)$. Without loss of generality we can assume $v_p(\eta) \le v_p(\mu)$. The change of basis $$u=x-\frac{\mu}{\eta}z, \quad v=y+\frac{\mu}{\eta} z, \quad t=z$$ yields 
\[
 [u,v]  = \mu u,\quad [v,t]  = \eta u,\quad [t,u]  = 0.
\]
and $L_2(\eta,\mu)\cong L_3(\mu,\eta)$. Hence $L_2(\eta,\mu) \cong L_{\mu,\eta,0}^{y,y,0} \cong L_{0,\eta,\mu}^{0,z,z} \cong L_3(\mu,\eta) $.
The calculations for the other cases are completely analogous and will be omitted.
\end{proof}
\end{lemma}

\begin{lemma}\label{lem:metabelianorSL21}
 Let $L=L_{\a,\b,\c}^{\xi_1,\xi_2,\xi_3}$ be the Lie algebra defined in Lemma~\ref{lem:xyz}. Then $L$ is not metabelian if, and only if, $L = L_{\eta,\eta,-2\rho}^{x,z,y}$. Moreover, in this case $L$ is commensurable to the Lie algebra of $\mathrm{SL}_2^1(\mathbb{Z}_p)$.

 \begin{proof}
  From the proof of Lemma~\ref{lem:xyz}, we can see that it is sufficient to check that $L_1(\eta,\mu,\lambda,\rho)$ and $L_3(\eta,\mu)$ are metabelian for every allowed choice of coefficients. It is clear that $L_3(\eta,\mu)$ is metabelian. 
  
  The derived subalgebra of $L_{0,\lambda,\eta}^{0,y,x}$ is generated by $\lambda y$ and $\eta x$ and these elements commute in $L_{0,\lambda,\eta}^{0,y,x}$.
  
  The last claim follows from the fact that $\mathbb{Q}_p \otimes L \cong \mathfrak{sl}_2(\mathbb{Q}_p)$ (see for instance \cite[Prop.~2.30]{ns:selfsim}). 
 \end{proof}
\end{lemma}
 
 \subsection{Solvable triangle groups}

We start by determining the isomorphism classes of $p$-RAAGs with exactly one edge. 

\begin{lemma}\label{lem:isom}
 Let $G_{\a,\b}$ be the pro-$p$ group defined by $$\pres{x,y}{[x,y]=x^\a y^\b}$$ with $\a,\b\in p^{1+\varepsilon}\mathbb{Z}$ ($\varepsilon =0$ if $p$ is odd and $\varepsilon=1$ if $p=2$) and set $c=\min \{v_p(\a),v_p(\b)\}$. Then $G_{\a,\b} \cong G_{p^{c},0}$.
\end{lemma}
\begin{proof}
 For $\a=\b=0$, the group $G_{0,0}$ is isomorphic to $\mathbb{Z}_p^2$ and the statement is clear. We can then suppose that $\a\neq 0$ and $v_p(\a) \le v_p(\b)\le \infty$. 
 
 Note that $G_{\a,\b}$ is a $2$-generated powerful pro-$p$ group. In particular, $[G_{\a,\b},G_{\a,\b}]$ is a finitely generated normal subgroup of $G_{\a,\b}$ and it is non-trivial. In fact, one can easily show that $$G_{\a,\b}/[G_{\a,\b},G_{\a,\b}] \cong \mathbb{Z}_p \times \left(\mathbb{Z}_p/p^{v_p(\a)}\mathbb{Z}_p\right).$$
Since $G_{\a,\b}$ has positive deficiency, \cite[Thm.~4]{HS:deficiency} yields that $G_{\a,\b}$ is a pro-$p$ duality group of dimension $2$. Hence, $G_{\a,\b}$ is a $p$-adic analytic Demushkin group of dimension $2$. By \cite[Prop.~7.1]{kgs:small}, $G_{\a,\b}$ is isomorphic to the group 
 \[
    H= \pres{x,y}{[x,y]= x^{p^k}}
 \]
for some positive integer $k$ ($k\ge 2$ for $p=2$). Comparing the abelianisations of $G_{\a,\b}$ and $H$, we conclude that $k= v_p(\a)$. 
\end{proof}

\begin{lemma}\label{lem:2gens}
 Let $\a,\b$ and $G_{\a,\b}$ be as above. Consider the powerful $\mathbb{Z}_p$-Lie lattice $L_{\a,\b}$ defined by $$\pres{x,y}{[x,y] = \a x+ \b y}$$ and the associated pro-$p$ group $G_{L_{\a,\b}}$ under the Lazard correspondence. Then 
 \[
   G_{L_{\a,\b}} \cong G_{\a,\b}.
 \]
\end{lemma}
\begin{proof}
 Without loss of generality we may suppose that $v_p(\a)\le v_p(\b) \le \infty$. We will first perform a change of basis in $L_{\a,\b}$: set $u = x+ (\b/\a) y$ and $v= y$, then $L_{\a,\b}' = \pres{u,v}{[u,v] = \a u}_{\mathrm{Lie}}$ is clearly isomorphic to $L_{\a,\b}$. 
 
 Let $a=p^{v_p(\a)}$. Working directly with the Lie bracket definition of the uniform $p$-adic analytic pro-$p$ group $H_{a} = \pres{x,y}{[x,y]=x^{a}}$, it is easy to see that its associated Lie algebra $L_{H_a}$ is isomorphic to $\pres{x,y}{[x,y]= a x}_{\mathrm{Lie}}$. By a standard Lie-theoretic computation, it follows that $L_{\a,\b}'$ is isomorphic to $L_{H_a}$. Therefore $G_{L_{\a,\b}}  \cong G_{L_{H_{a}}}$. By the Lazard correspondence $G_{L_{H_{a}}} \cong H_a$ and finally $H_a \cong G_{\a,\b}$, by Lemma~\ref{lem:isom}. 
\end{proof}

In a similar fashion to the previous section, we are going to define some groups that will turn out to correspond to the above Lie algebras. For $\a,\b,\c \in p\mathbb{Z}_p$, and $\xi_1\in \{0,x,y\}$, $\xi_2\in \{0,y,z\}$ and $\xi_3\in \{0,z,x,y\}$, define the pro-$p$ group
\[
   G_{\a,\b,\c}^{\xi_1,\xi_2,\xi_3} = \pres{x,y,z}{[x,y] = \xi_1^\a,\ [y,z] = \xi_2^\b,\ [z,x] = \xi_3^\c}
\] 
Also define $$G_2(\a,\b) = \pres{x,y,z}{[x,y] = 1,\ [y,z] = y^\a z^\b,\ [z,x] = 1}.$$ It turns out that the picture for groups is analogous to that of Lie lattices (cf.\ Proposition~\ref{prop:families}).

\begin{lemma}\label{lem:3gens}
The groups $G_2(\a,\b)$, $G_{0,\b,\c}^{0,z,z}$ and $G_{0,\b,\c}^{0,y,x}$ are metabelian uniform pro-$p$ groups of dimension $3$ for every choice of parameters. Moreover, 
\[
  L_{G_2(\a,\b)} \cong L_2(\a,\b)\cong L_{0,\b,\a}^{0,z,z}, \quad L_{G_{0,\b,\c}^{0,z,z}} \cong L_{0,\b,\c}^{0,z,z} \text{\ \ and \ \ } L_{G_{0,\b,\c}^{0,y,x}} \cong L_{0,\b,\c}^{0,y,x}.
\]
In particular, $G_{0,\b,\c}^{0,z,z} \cong G_2(\b,\c)$.
\end{lemma}
\begin{proof}
First of all notice that these groups are powerful pro-$p$ groups by definition. 
  We first consider $G_{0,\b,\c}^{0,y,x}$. Define the homomorphisms
    \[
       \varphi_1: G_{0,\b,\c}^{0,y,x} \to \pres{z,x}{[z,x] = x^\c} \text{ and } \varphi_2: G_{0,\b,\c}^{0,y,x} \to \pres{y,z}{[y,z] = y^\b}. 
    \]
    By Lemma~\ref{lem:2gens}, the subgroup generated by $x$ and $z$ is uniform of dimension $2$. Moreover, the kernel of $\varphi_1$ is generated by $y$ and hence infinite, because $y$ has infinite image via $\varphi_2$. In particular, the dimension of $G_{0,\b,\c}^{0,y,x}$ as a $p$-adic analytic group must satisfy $$ 3\ge \dim(G_{0,\b,\c}^{0,y,x}) \ge \dim (\mathrm{Im}\ \varphi_1 ) + \dim (\mathrm{Ker}\ \varphi_1 ) = 2+1 =3. $$ Now, $G_{0,\b,\c}^{0,y,x}$ is a powerful pro-$p$ group with $d(G)=\dim(G)$ and, by \cite[Prop.~2.12]{ks}, it must be uniform.  The isomorphism of Lie algebras now follows from the definition of the Lazard Lie bracket via a straight-forward calculation using \eqref{eq:Liebracket}.
    
  For the group $G_2(\a,\b)$ the proof is similar to the previous case using the homomorphisms induced by $x\mapsto 1$ and $y\mapsto 1$. 
 Moreover, it is clear that $G_2(\a,\b) = \langle x \rangle \times \langle y,z \rangle$. By Lemma~\ref{lem:2gens} applied to the subgroup generated by $y$ and $z$, we deduce that $L(G_2(\a,\b)) \cong \mathbb{Z}_p \oplus L_{\a,\b} \cong L_2(\a,\b)$. 

We need to use a different strategy to prove the claims about $G_{0,\b,\c}^{0,z,z}$. Without loss of generality, we can assume that $v_p(\b)\ge v_p(\a)$. For $\delta \in \mathbb{Z}_p$, since $y^{-1}z y = z^{1-\a}$, we deduce that $[z, y^\delta] = z^{(1-\a)^\delta -1}$. Hence, $$ [y^\delta x, z] = [y^\delta,z]^x [x,z] =  (z^{1-(1-\a)^\delta})^x z^{-\b} = z^{(1+\b)(1-(1-\a)^\delta)} z^{-\b}.$$ By setting $[z,y^\delta]=1$, we obtain the equation $ (1+\b)(1-(1-\a)^\delta) = \b. $ Solving for $\delta$ we obtain $\delta = \log(1-\frac{\b}{1+\b})/\log(1-\a) \in \mathbb{Z}_p,$ since $v_p(\b)\ge v_p(\a)$. Therefore $$ G_{0,\b,\c}^{0,z,z} = \pres{y^\delta x, y, z}{[y^\delta x, y] =1,\ [y,z] = z^\a,\ [z,y^\delta x] =1} \cong G_2(0,\a).$$

In conclusion, the isomorphisms of Lie algebras are obtained using the definition of Lie bracket \eqref{eq:Liebracket} and comparing abelianisations.
\end{proof}

We are now ready to prove Theorem~\ref{thm:quadratictrianglepraags}.

\begin{proof}[\textbf{Proof of Theorem~\ref{thm:quadratictrianglepraags}}]
Note that the quadratic triangle $p$-RAAG $G_A$ is powerful. Hence, $G_A$ must be uniform of dimension $3$, by Theorem~\ref{thm:quadratic analytic}. Since $G_A$ is uniform, we can consider the associated Lazard Lie algebra. The result now follows from Lemma~\ref{lem:xyz}, Lemma~\ref{lem:metabelianorSL21} and Lemma~\ref{lem:3gens}.
\end{proof}

To the interested reader the last theorem might sound unsatisfactory, as we have a pretty clear picture for the solvable case and not so for the non-solvable one. 
In the next section, we will content ourselves with the computation of a somewhat special case to outline the general method that could be used to produce many non-solvable triangle groups.

\subsection{Unsolvable triangle groups}

In this section we will use the methods of \cite{ac} to study non-solvable triangle $p$-RAAGs. 

Let $G=G(\eta,\mu,\lambda)$ be the pro-$p$ group defined by the balanced pro-$p$ presentation
\begin{equation}\label{eq:pres}
\pres{x,y,z}{[x,y]= x^\eta y^\mu,\  [y,z] =  y^\lambda  z^\eta,\ [z,x] = z^\mu  x^\lambda }
\end{equation} with $\eta,\mu, \lambda \neq 0$. Suppose additionally that $v_p(\lambda) = v_p(\mu) = v_p(\eta)=1$. We will show that $G$ is isomorphic to $\mathrm{SL}_2^1(\mathbb{Z}_p)$ by proving that the latter admits a presentation of the form \eqref{eq:pres}. 

Given a pro-$p$ group $G$, we can form its \emph{graded Lie algebra}: for $n\ge 1$, define
$$\mathfrak{gr}_n(G) = P_n(G)/P_{n+1}(G) \text{ and } \mathfrak{gr}(G) = \bigoplus_{n\ge 1} \mathfrak{gr}_n (G). $$
Denote by $\phi_n: P_n(G) \to \mathfrak{gr}_n (G)$ the quotient map.
It is straightforward to check that the maps $\pi_n : \mathfrak{gr}_n (G) \to \mathfrak{gr}_{n+1} (G)$ induced by $p$-powers
are linear and these extend uniquely to the linear map $\pi : \mathfrak{gr} (G)\to \mathfrak{gr} (G)$.
Finally, commutators in the graded components $\mathfrak{gr}_n (G)$ endow $\mathfrak{gr} (G)$ with the structure of
a graded Lie $\mathbb{F}_p [\pi]$-algebra.

We will need the following 
two lemmas from \cite{ac} which can also be checked directly.

\begin{lemma} 
  The Lie $\mathbb{F}_p [\pi]$-algebra $\mathfrak{gr} (G)$ is generated by its homogeneous component $\mathfrak{gr}_1 (G)$; moreover $\mathfrak{gr} (G)$ is generated by $\{\phi_1(x) \vert x\in X \}$ for every generating set $X$ of $G$.
 \end{lemma}
 \begin{lemma}
  The graded Lie algebra $\mathfrak{gr}(\mathrm{SL}_2^1(\mathbb{Z}_p))$ of $\mathrm{SL}_2^1(\mathbb{Z}_p)$ can be presented by 
  \begin{equation}\label{eq:sl21}
   \pres{\ol{x}_1,\ol{x}_2,\ol{x}_3}{[\ol{x}_1,\ol{x}_2] = \pi \ol{x}_1, [\ol{x}_2,\ol{x}_3] = \pi \ol{x}_3, [\ol{x}_3,\ol{x}_1] = \pi \ol{x}_2}.
  \end{equation}
In particular,  it is a free $\mathbb{F}_p [\pi]$-module of rank $3$.
 \end{lemma}

\begin{proposition}\label{prop:SL21triangle}
 The group $\mathrm{SL}_2^1(\mathbb{Z}_p)$ has a pro-$p$ presentation $G(\eta,\mu,\lambda)$ of the form \eqref{eq:pres} 
 for every $\eta,\mu,\lambda\in p\mathbb{Z}_p$ and $v_p(\eta)= v_p(\mu)= v_p(\lambda)=1$.
 \begin{proof}
 We will write $G=\mathrm{SL}_2^1(\mathbb{Z}_p)$. We proceed by induction. Suppose that, for $n\ge 1$, we defined three generators $x_{n},y_{n},z_{n} \in G$ and three coefficients $\eta_{n},\mu_{n},\lambda_{n}\in \mathbb{Z}_p$ such that these satisfy the relations $\mathcal{R}$ modulo $P_{n+2}(G)$, that is  \begin{equation} \label{eq:xnynzn} [x_{n},y_{n}] \equiv x_{n}^{\eta_{n}} y_{n}^{\mu_{n}},\quad [y_{n},z_{n}] \equiv y_{n}^{\lambda_{n}} z_{n}^{\eta_{n}}, \quad [z_{n},x_{n}] \equiv z_{n}^{\mu_{n}} x_{n}^{\lambda_{n}} \mod P_{n+2}(G)
 \end{equation}
 \begin{equation} 
   \eta_{n} \equiv \eta,\quad  \mu_{n} \equiv \mu,\quad \lambda_{n} \equiv \lambda \mod p^{n+1}\mathbb{Z}_p. \end{equation} By the proof of Lemma~\ref{lem:xyz} (or by base-change), the Lie $\mathbb{F}_p$-algebra $\mathfrak{gr}_1(G(\eta,\mu,\lambda))$ admits a presentation of the form \eqref{eq:sl21}, hence we have the case $n=1$. 
  
Let $\d_1,\d_2,\d_3 \in \{0,\ldots, p-1\}$ be the integers such that 
\begin{equation}
\delta_1 \equiv \eta -\eta_n,\quad \delta_2 \equiv \mu -\mu_n,\quad \delta_3 \equiv \lambda -\lambda_n \mod p^{n+2}\mathbb{Z}_p
\end{equation}
and define the new coefficients $\eta_{n+1}= \eta_{n} + p^{n+1} \d_1$, $\mu_{n+1}= \mu_{n} + p^{n+1} \d_2$ and $\lambda_{n+1}= \lambda_{n} + p^{n+1} \d_3$.

Let $\D_1,\D_2,\D_3 \in P_{n+1}(G)$ and define the new elements $x_{n+1} = x_{n} \D_1$, $y_{n+1}= y_n \D_2$ and $z_{n+1}= z_n \D_3$. We will prove that we can choose $\D_1,\D_2,\D_3$ so that 
 \begin{align*}
     f_1 &=  [x_{n+1},y_{n+1}] (x_{n+1}^{\eta_{n+1}} y_{n+1}^{\mu_{n+1}})^{-1} \cdot ( [x_{n},y_{n}] (x_{n}^{\eta_{n}} y_{n}^{\mu_{n}})^{-1} )^{-1} \\
     f_2 &=  [y_{n+1},z_{n+1}] (y_{n+1}^{\lambda_{n+1}} z_{n+1}^{\eta_{n+1}})^{-1} \cdot ( [y_{n},z_{n}] (y_{n}^{\lambda_{n}} z_{n}^{\eta_{n}})^{-1} )^{-1} \\
     f_3 &=  [z_{n+1},x_{n+1}] (z_{n+1}^{\mu_{n+1}} x_{n+1}^{\lambda_{n+1}})^{-1} \cdot ( [z_{n},x_{n}] (z_{n}^{\mu_{n}} x_{n}^{\lambda_{n}})^{-1} )^{-1}
   \end{align*}
 are congruent to $1$ modulo $P_{n+3}(G)$ for every $\d_1,\d_2,\d_3 \in \mathbb{Z}_p$. This will allow us to define the new generators $x_{n+1}$, $y_{n+1}$ and $z_{n+1}$ with the required properties. Finally, the elements $x= \lim_{n\to \infty} x_n$, $y= \lim_{n\to \infty} y_n$ and $z= \lim_{n\to \infty} z_n$ clearly deliver a presentation of the form $G(\eta,\mu,\lambda)$.
 
 First we note that $f_i \in P_{n+2}(G)$, $i=1,2,3$. Now, denote by $\ol{g}$ the image of $g\in G$ in $\mathfrak{gr}_1(G)$ and write $\ol{x}=\phi_1(x_n)$, $\ol{y}=\phi_1(y_n)$, $\ol{z}=\phi_1(z_n)$. Thus $\ol{\D_i}$ can be written as $\pi^n ( \c_{i1} \ol{x}+ \c_{i2} \ol{y} + \c_{i3} \ol{z})$ for some $\c_{ij}\in \mathbb{F}_p$. Reducing $f_i$ modulo $P_{n+3}(G)$, it is easy to show that 
   \begin{align*}
     \ol{f}_1 &= \pi^{n+1} ( [\ol{x},\ol{\D}_2] + [\ol{\D}_1,\ol{y}] - \d_1 \ol{x} - \d_2 \ol{y} ) \\
     \ol{f}_2 &= \pi^{n+1} ( [\ol{y},\ol{\D}_3] + [\ol{\D}_2,\ol{z}] - \d_3 \ol{y} - \d_1 \ol{z} ) \\
     \ol{f}_3 &= \pi^{n+1} ( [\ol{z},\ol{\D}_1] + [\ol{\D}_3,\ol{x}] - \d_2 \ol{z} - \d_3 \ol{x} ).
   \end{align*}
Using the relations \eqref{eq:xnynzn} and the expressions of $\ol{\D}_i$, we obtain the equations
   \begin{align*}
      (\c_{22} + \c_{11} -\c_{23} -\d_1) \ol{x} + (\c_{22} + \c_{11} -\c_{13} -\d_2) \ol{y} +(-\c_{23} -\c_{13}) \ol{z} &=0 \\
      (-\c_{31} -\c_{21}) \ol{x} + (-\c_{31} +\c_{33} +\c_{22} -\d_3) \ol{y} + (\c_{33} - \c_{21} + \c_{22} -\d_1) \ol{z} &=0 \\
      (\c_{11} -\c_{32} + \c_{33} -\d_3) \ol{x} + (-\c_{12} - \c_{32} ) \ol{y} + (\c_{11} - \c_{12} + \c_{33} -\d_2) \ol{z} &=0.
   \end{align*} 
Therefore we have to solve the $9\times 9$ system of equations obtained by equating all coefficients to $0$ with the parameters $\d_1,\d_2,\d_3$. This system can be represented by
$$  \begin{bmatrix}
   1 & 0 & 0 & 0 &1 & -1 & 0 & 0 & 0 \\
   1 & 0 & -1 & 0 & 1 & 0 & 0 & 0 & 0 \\
   0 & 0 & -1 & 0 & 0 & -1 & 0 & 0 & 0 \\
   0 & 0 & 0 & -1 & 0 & 0 & -1 & 0 & 0 \\
   0 & 0 & 0 & 0 & 1 & 0 & -1 & 0 & 1 \\
   0 & 0 & 0 & -1 & 1 & 0 & 0 & 0 & 1 \\
   1 & 0 & 0 & 0 & 0 & 0 & 0 & -1 & 1 \\
   0 & -1 & 0 & 0 & 0 & 0 & 0 & -1 & 0 \\
   1 & -1 & 0 & 0 & 0 & 0 & 0 & 0 &1 \\ 
  \end{bmatrix} 
  \begin{bmatrix}
\c_{11} \\
\c_{12} \\
\c_{13} \\
\c_{21} \\
\c_{22} \\
\c_{23} \\
\c_{31} \\
\c_{32} \\
\c_{33} \\
\end{bmatrix}
=
  \begin{bmatrix}  
  \d_1 \\
  \d_2 \\
  0 \\
  0 \\
  \d_3 \\
  \d_1 \\
  \d_3 \\
  0 \\
  \d_2
  \end{bmatrix}
$$

One can easily check that this system has the solution 
\begin{equation*}
  \frac{1}{2} ( {\d_2},\  -(\d_2-\d_3),\  \d_1-\d_2,\  -(\d_1-\d_3),\  \d_1,  -(\d_1-\d_2),\  \d_1-\d_3,\  \d_2-\d_3,\  \d_3)
\end{equation*}
 \end{proof}
\end{proposition}

\begin{remark}
 Note that the $9\times 9$ matrix appearing above has rank $8$.  
\end{remark}

Finally, in order to get a large family of subgroups of $\mathrm{SL}_2^1(\mathbb{Z}_p)$ as triangle groups, one can consider the subgroups $\gen{x^{p^{a_1}}, y^{p^{a_2}}, z^{p^{a_3}}}$ for $a_i\ge 1$ inside $G(\eta,\mu,\lambda)$ with $v_p(\eta)=v_p(\mu)=v_p(\lambda)=1$. These are powerful pro-$p$ groups and they satisfy relations of the form
 \begin{equation*}
    [x^{p^{a_1}},y^{p^{a_2}}]  = (x^{p^{a_1}})^{\a_1} (x^{p^{a_1}})^{\a_2},\quad
    [y^{p^{a_2}},z^{p^{a_3}}]  = (y^{p^{a_2}})^{\b_2} (z^{p^{a_3}})^{\b_3},
    \end{equation*}
    \begin{equation*}
    [z^{p^{a_3}},x^{p^{a_1}}]  = (z^{p^{a_3}})^{\c_3} (x^{p^{a_1}})^{\c_1}
 \end{equation*}
 for some $\a_1,\a_2,\b_2,\b_3,\c_1,\c_3 \in p\mathbb{Z}_p$. 

 It is possible to slightly modify the proof of Proposition~\ref{prop:SL21triangle} to show that $\mathrm{SL}_2^k(\mathbb{Z}_p)$ admits a presentation $$G(\eta,\mu,\lambda)= \pres{x,y,z}{[x,y]= x^\eta y^\mu,\  [y,z] =  y^\lambda  z^\eta,\ [z,x] = z^\mu x^\lambda}$$ with $\eta,\mu,\lambda\in p\mathbb{Z}_p$ and $v_p(\eta)= v_p(\mu)= v_p(\lambda)=k$ for $k\ge 2$.
 
 We believe that all the groups of the form \eqref{eq:pres} are torsion-free and therefore uniform for any choice of parameters. Note that, by \cite{ac} and \cite[Prop.~2.12]{ks}, it would be sufficient to show that these groups are infinite.

\section{Non-abelian free subgroups}\label{sec:free subgp}

In Galois theory one has the following version of the celebrated \emph{Tits alternative} (cf.\ \cite{cq:bk} and \cite[Thm.~3]{ware}).

\begin{theorem}\label{thm:tits}
Let $K$ be a field containing a root of unity of order $p$.
Then either $G_{K}(p)$ is (metabelian) uniform, or it contains a free non-abelian pro-$p$ subgroup.
\end{theorem}

One may ask whether a similar result holds also for quadratic pro-$p$ groups which do not arise as maximal 
pro-$p$ Galois groups of fields. In the case of $p$-RAAGs we have Theorem~\ref{thm:pRAAGs free subgroup}, which we prove next.

\subsection{Non-abelian free subgroups in \texorpdfstring{$p$}{p}-RAAGs}

\begin{proof}[\textbf{Proof of Theorem~\ref{thm:pRAAGs free subgroup}}]
 Let $G_\Gamma$ be a $p$-RAAG with associated $p$-graph $\Gamma=(\calG,f)$ and underlying graph $\calG=(\sV,\sE)$. Assume that $G_\Gamma$ is not powerful.
 Set $d=\dd(G)$ and let $$G_\Gamma=\langle x_1,\ldots,x_d\mid R\rangle$$ be the presentation induced by $\Gamma$.
 By Proposition~\ref{lem:unif}~(b), $\calG$ is not complete, hence there exist
 two vertices in $\sV(\calG)$ --- say $x_1$ and $x_2$ --- such that $(x_1,x_2)\notin\sE(\calG)$.
 
Consider the pro-$p$ group $S=\langle a,b\mid a^p=b^p=1 \rangle$. Note that $S\cong C_p\ast C_p$. Let $\phi\colon G\to S$
be the homomorphism defined by $\phi(x_1)=a$, $\phi(x_2)=b$ and $\phi(x_i)=1$ for $i\neq1,2$.

Set $H=\langle x_1, x_2\rangle$.
Then $\phi|_H\colon H\to S$ is an epimorphism.
The elements $t_1=aba$ and $t_2=bab$ in $S$ have infinite order and the subgroup they generate is not pro-cyclic.
Thus, $\langle t_1,t_2\rangle$ is a 2-generated free pro-$p$ group.
Now choose two elements $u,v\in H$ such that $\phi(u)=t_1$ and $\phi(v)=t_2$.
Then $\phi|_{\langle u,v\rangle}:\langle u,v\rangle\to\langle t_1,t_2\rangle$
is an epimorphism, and from the hopfian property it follows that $\langle u,v\rangle$ is a 2-generated free 
pro-$p$ group.
\end{proof}

\begin{remark}
Let $G=G_1\amalg_{H}G_2$ be a pro-$p$ group as in Theorem~\ref{thm:cohomology amalgam} and suppose that $H$ is not equal to $G_1$ or $G_2$, i.e., the amalgam is non-fictitious.   Then the standard graph $S=S(G)$ of $G$ is defined as follows (cf. \cite{ribzal:horizons}): 
 $$S = G/H \bigcupdot  G/G_1 \bigcupdot G/G_2,\ V(S) =  G/G_1 \bigcupdot G/G_2,\ d_0(gH)=gG_1,\ d_1(gH)=gG_2.$$ 
 Now let $G=\mathrm{HNN}(G_0,A,\phi) = \pres{G_0, t}{tat^{-1}=\phi(a), a\in A}$  be as in Theorem~\ref{thm:HNN}. Then the standard graph $S=S(G)$ of $G$ is defined as follows:
 $$S = G/A \bigcupdot  G/G_0,\quad V(S) =  G/G_0,\quad d_0(gA)=gG_0,\quad d_1(gH)=gtG_0.$$ 
 In both cases $S$  is a pro-$p$ tree (see  \cite[Thm.~4.1]{ribzal:horizons}). Moreover,  it is not difficult to see that $G$ acts faithfully and irreducibly on $S$. By  \cite[Thm.~3.15]{ribzal:horizons}, if $G$ is not isomorphic to  $C_2\oplus\Z_2$ or $\Z_p$, then $G$ conatins a free non-abelian pro-$p$ subgroup.
 
\end{remark}

As we have seen in the previous proof, in a quadratic $p$-RAAG which is not uniform there must be a ``missing'' commutator among its relations. It is interesting to remark that the same is true for quadratic \emph{mild} pro-$p$ groups.

\begin{proposition}\label{prop:mild cupproduct}
Let $G$ be a mild quadratic pro-$p$ group with $\rr(G)\geq\dd(G)$. 
Then there exist linearly independent elements $\alpha,\alpha'\in H^1(G,\F_p)$ such that $\alpha \alpha'=0$.
\end{proposition}

\begin{proof}
Let \eqref{eq:presentation} be a minimal presentation of $G$. Let $\mathcal{X}=\{x_1,\ldots,x_d\}$ be a basis of $F$
and let $\{\alpha_1,\ldots,\alpha_d\}$ be a basis of $H^1(G,\F_p)$ dual to $\mathcal{X}$.

Suppose that $\alpha \alpha'\neq0$ for every linearly independent couple $\alpha,\alpha'\in H^1(G,\F_p)$.
Then, by bilinearity of the cup-product, $\{\alpha_1 \alpha_h \mid h=2,\ldots,d\}$ is a set of linearly independent
elements of $H^2(G,\F_p)$. This implies that, for every $h=1,\ldots,d-1$, the commutator $[x_1,x_{h+1}]$ appears in some relation. Moreover $\rr(G):=m\geq d> d-1$.
By Remark~\ref{rem:Gauss reduction} and the above discussion, one may pick a set of defining relations
$\mathcal{R}=\{r_1,\ldots,r_m\}$ such that
\begin{equation}\label{eq:rels1} 
   r_h\equiv  [x_1,x_{h+1}]\cdot \displaystyle\prod_{2\leq i<j\leq d}[x_i,x_j]^{b(h,i,j)} \mod F_{(3)},\quad \text{ for } 1\le h\le d-1 
\end{equation}
and
\begin{equation}\label{eq:rels2}
     r_h\equiv \displaystyle\prod_{2\leq i<j\leq d}[x_i,x_j]^{b(h,i,j)} \mod F_{(3)}, \quad \text{ for } h\geq d
\end{equation}
for appropriate coefficients $b(h,i,j) \in \mathbb{F}_p$. Let $\widetilde{G}$ be the pro-$p$ group with presentation $\langle \mathcal{X}\mid \widetilde{\mathcal{R}} \rangle$ where $\widetilde{\mathcal{R}}=\{\widetilde{r}_1,\ldots,\widetilde{r}_m\}$ and $\widetilde{r}_h$ is the coset representative of $r_h$ modulo $F_{(3)}$ appearing in the righ-hand side of \eqref{eq:rels1} and \eqref{eq:rels2}. Since $G$ is mild, by Remark~\ref{rem:mild approximation} $\widetilde{G}$ is also mild. Let $N,H\le \widetilde{G}$ be the subgroups generated by $\{x_2,\ldots,x_d\}$ and $\{x_1\}$, respectively.
Then $N$ is a normal subgroup of $\widetilde{G}$. Thus the short exact sequence of pro-$p$ groups  
\[
 \xymatrix{1\ar[r] & N\ar[r] & \widetilde{G}\ar[r] & H \ar[r] & 1},
\]
induces the short exact sequences in cohomology
\begin{equation}\label{eq:ses cohom GGS}
  \xymatrix{ 0\ar[r] & H^1(H,H^{n-1}(N,\F_p))\ar[r] & H^n(\widetilde{G},\F_p)\ar[r] & H^n(N,\F_p)^{\widetilde{G}}\ar[r] &0}
\end{equation}
for every $n\geq1$ (cf. \cite[\S~II.4, Ex.~4]{nsw:cohn}).
In particular, from \eqref{eq:ses cohom GGS} for $n=2$ we can deduce that $H^2(N,\F_p)^{\widetilde{G}}\neq0$.
Therefore $H^1(H,H^{2}(N,\F_p))\neq0$. Finally \eqref{eq:ses cohom GGS} for $n=3$ yields $H^3(\widetilde{G},\F_p)\neq0$,
contradicting $\cdd(\widetilde{G})=2$.
\end{proof}

\subsection{Non-abelian free subgroups in mild pro-\texorpdfstring{$p$}{p} groups}

Even with Proposition~\ref{prop:mild cupproduct} in hand, we were not able to prove an analogous of Theorem~\ref{thm:pRAAGs free subgroup} for mild pro-$p$ groups in full generality. One reason for this is that the condition of mildness for a pro-$p$ group $G$ depends only on the shape of the defining relations modulo $G_{(3)}$ (cf.\ Remark~\ref{rem:mild approximation}). 

Nevertheless, we can show that many mild pro-$p$ groups contain a free non-abelian subgroup. To do this, we will show that ---in several cases--- a mild pro-$p$ group is a \emph{generalised Golod-Shafarevic group} (see Section~\ref{sec:ggs}).

\begin{proposition}\label{prop:mild GoSha}
 Let $G$ be a mild quadratic and non-uniform pro-$p$ group
and let \eqref{eq:presentation} be a minimal presentation of $G$. Choose a basis $\mathcal{X}=\{x_1,\ldots,x_d\}$ of $F$
and a set of defining relations $\mathcal{R}=\{r_1,\ldots,r_m\}\subset F$ for $G$. 
Suppose that one of the following conditions holds:
\begin{itemize}
 \item[(a)] $\rr(G) \neq \dd(G)^2/4$;
 \item[(b)] there are $x_i,x_j\in\mathcal{X}$, $i\neq j$, such that $x_i^3,x_j^3$, $[x_i,x_j]$ and any other higher commutator involving only $x_i$ and $x_j$ do not appear in any defining relation $r_h\in\mathcal{R}$;
 \item[(c)] every defining relation consists of a single elementary commutator modulo $F_{(3)}$, i.e., 
 $r_h\equiv [x_{i_h},x_{j_h}]\bmod F_{(3)}$ for some $1\leq i_h<j_h\leq d$, for all $r_h\in\mathcal{R}$.
\end{itemize}
Then $G$ is generalised Golod-Shafarevic. In particular, it contains a free non-abelian pro-$p$ subgroup. 
\end{proposition}

\begin{proof}
Note that, by Theorem~\ref{thm:ershovfreesub}, it is sufficient to show that $G$ is a generalised Golod-Shafarevic pro-$p$ group.  Set $d=\dd(G)$ and $m=\rr(G)$. By \cite[Prop.~4]{labute:fabulous}, 
\begin{equation}\label{eq:gocha}
m \le \frac{d^2}{4}. 
\end{equation}
If \eqref{eq:gocha} is a strict inequality, then $G$ is a Golod-Shafarevich pro-$p$ group by \cite{zelmanov}.
Thus it is also generalized Golod-Shafarevich.
This settles part (a).

If \eqref{eq:gocha} is an equality, then $d=2n$ for some $n\geq2$.
Without loss of generality, suppose that condition (b) holds with $i=1$, $j=2$.
We define a valuation $D$ on $F$ as follows: set $D(x_i)=1$, for $i=1,2$, and $D(x_i)=N$, with $N$ to be chosen later.
It is easy to show that $D$ comes from a weight function on $\F_p\langle\!\langle\mathcal{X}\rangle\!\rangle$
and that $D(x_i^p)=p D(x_i)$ and $D([x_i,x_j])=D(x_i)+D(x_j)$, for $1\le i<j\le d$.
By (b), if $D(r_h)\le 5$ for some $h$, then one has $p=5$ and $r_h=x_1^{ap}x_2^{bp}r_h'$, with $a,b\in\{0,1\}$ not both equal to 0, and $D(r_h')\geq 1+N$.
Thus one may choose $\mathcal{R}$ such that $x_1^p$ and $x_1^p$ possibly appear only in at most two relations,
say $r_1$ and $r_2$.
Hence $D(r_1),D(r_2)\geq 5$ and $D(r_h)\geq1+N$, for $h\geq3$.
It follows that
\[\begin{split}
   1-H_{\mathcal{X},D}(T)+H_{\mathcal{R},D}(T)&=1-(2T+(d-2)T^N)+(T^{D(r_1)}+T^{D(r_2)}+\mathcal{O}(T^N))\\ 
   &\leq 1-2T+2T^5+\mathcal{O}(T^{N-1})
  \end{split}\]
  for $T\in (0,1)$.
Since $N$ can be chosen arbitrarily large, there exists $T_0\in(0,1)$ such that
$1-H_{\mathcal{X},D}(T_0)+H_{\mathcal{R},D}(T_0)<0$.
Hence $G$ is generalized Golod-Shafarevich, and this settles part (b).

Finally, suppose condition (c) holds.
Let $\calG=(\sV,\sE)$ be the combinatorial graph with 
$\sV=\mathcal{X}$ and $\sE=\{(x_{i_h},x_{j_h}),h=1,\ldots,m\}$.
Since $cd(G)=2$ and $G$ is quadratic, the graph $\calG$ is triangle-free. By Mantel's Theorem (cf. \cite{mantel}),
$\calG$ is a complete bipartite graph on $2n$ vertices and $n^2$ edges.
Thus, after renumbering we can arrange that the couples $(i_h,j_h)$ are all the couples with $1\leq i_h,j_h\leq d$, $i_h$ odd and $j_h$ even.
We define a valuation $D$ on $F$ as follows: set $D(x_i)=1$, for $i$ odd, and $D(x_j)=2$ for $j$ even.
It is easy to show that $D$ comes from a weight function on $\F_p\langle\!\langle\mathcal{X}\rangle\!\rangle$
and that $D([x_i,x_j])=D(x_i)+D(x_j)$, for $1\le i<j\le d$.
Hence $D(r_h)=3$ for every $h=1,\ldots,m$.
It follows that
\[\begin{split}
   1-H_{\mathcal{X},D}(T)+H_{\mathcal{R},D}(T)&=1-\left(nT+nT^2\right)+n^2T^3\\ 
   &= \left(1-nT\right)\left(1-nT^2\right).
  \end{split}\]
Thus, there exists $T_0\in(1/n,1/\sqrt{n})$ such
that $1-H_{\mathcal{X},D}(T_0)+H_{\mathcal{R},D}(T_0)<0$.
Hence $G$ is generalized Golod-Shafarevich.
\end{proof}

 We conclude this section by showing that all ``small'' quadratic groups are either uniform or contain a free pro-$p$ subgroup.

\begin{coro}\label{coro:small free subgp}
 Let $G$ be a quadratic pro-$p$ group with $\dd(G)\leq3$.
 Then either $G$ is uniform, or $G$ contains a free non-abelian pro-$p$ subgroup.
\end{coro}
\begin{proof}
By Proposition~\ref{prop:mild GoSha}, it is enough to show that every quadratic pro-$p$ group with $\dd(G)\leq3$ which is not analytic is either mild or free.
This is clear for $\dd(G)\leq2$. For $\dd(G)=3$ and $\rr(G)=1$ it follows from \cite{cq:onerel}.

If $\dd(G)=3$ and $\rr(G)=2$ we may use Proposition~\ref{prop:cd2 cupproduct}.
In this case $\dim H^2(G,\F_p)=2$, so $H^2(G,\F_p)$ is generated by, say, the products $\alpha_1\alpha_2$
and $\alpha_1\alpha_3$. Then the subspaces $V_1=\langle\alpha_1\rangle$ and $V_2=\langle\alpha_2,\alpha_3\rangle$ of $H^1(G,\F_p)$ satisfy the hypotheses of Proposition~\ref{prop:cd2 cupproduct} and $G$ is mild.
\end{proof}

In light of the previous results, Conjecture~\ref{conj:subgp free} is very natural. Notice that Conjecture~\ref{conj:subgp free} holds for almost all known examples of quadratic groups: maximal pro-$p$ Galois groups, $p$-RAAGs, many mild pro-$p$ groups and ``small'' quadratic pro-$p$ groups.


\section*{}
\subsection*{Acknowledgments}
We would like to thank Thomas Weigel for inspiring us to work on quadratic pro-$p$ groups and for his continuous encouragement.
The second and third author thank the Heinrich-Heine University of D\"usseldorf and the University of Milan-Bicocca for their hospitality and support.

\bibliography{quad}

\bibliographystyle{amsplain}

\end{document}